\documentclass[11pt, eqno]{article}
\usepackage{bbm}
\usepackage{mathrsfs}
\usepackage{amsfonts}
\usepackage{amssymb}
\usepackage{graphicx}
\usepackage[all]{xy}
\usepackage{amsthm}
\usepackage{amsmath}
\usepackage{amsmath,amssymb,latexsym,color}
\usepackage[mathscr]{eucal}
\usepackage{CJK}
\usepackage{cases}
\usepackage{graphics}
\usepackage{citeref}

\usepackage{psfrag}
\usepackage{subfigure}
\usepackage{color}

\usepackage{amssymb,latexsym}
\usepackage{amsmath,latexsym}
\usepackage{amscd}

\newcommand{\R}{{\mathbb R}}

\newtheorem{theorem}{Theorem}[section]

\newtheorem{corollary}[theorem]{Corollary}

\newtheorem{thm}[theorem]{Theorem}
\newtheorem{definition}[theorem]{Definition}

\newtheorem{remark}[theorem]{Remark}

\newtheorem{lemma}[theorem]{Lemma}

\newtheorem{prop}[theorem]{Proposition}
\newtheorem{proposition}[theorem]{Proposition}

\begin{document}

\title{A Brunn-Minkowski type inequality for extended symplectic capacities
of convex domains and length estimate for a class of billiard trajectories}

\date{February 23, 2023}
\author{Rongrong Jin and Guangcun Lu
\thanks{Corresponding author
\endgraf \hspace{2mm} Partially supported
by the NNSF  11271044 of China.
\endgraf\hspace{2mm} 2010 {\it Mathematics Subject Classification.}
 53D35, 53C23 (primary), 70H05, 37J05, 57R17 (secondary).
 \endgraf\hspace{2mm} {\it Key words and phrases}. Extended Ekeland-Hofer-Zehnder symplectic capacities,
 Brunn-Minkowski type inequality, non-periodic  billiards, convex domains.
 }}
 \maketitle \vspace{-0.3in}




\abstract{
In this paper, we firstly generalize the Brunn-Minkowski type inequality for Ekeland-Hofer-Zehnder
symplectic capacity of bounded convex domains established by Artstein-Avidan-Ostrover in 2008 to
 extended symplectic capacities of bounded convex domains constructed by authors based on a class of Hamiltonian non-periodic boundary value problems recently.
  Then we introduce a class of non-periodic  billiards in convex domains, and for them we prove
  some corresponding results to those for periodic  billiards in convex domains obtained
   by Artstein-Avidan-Ostrover in 2012.
} \vspace{-0.1in}
\medskip\vspace{12mm}

\maketitle \vspace{-0.5in}


\tableofcontents

\section{Introduction and main results}\label{section:1}
\setcounter{equation}{0}

Throughout this paper, a  compact, convex subset of $\mathbb{R}^m$
with nonempty interior is called a \textsf{convex body} in $\mathbb{R}^m$.
The set of all convex bodies in $\mathbb{R}^m$ is denoted by $\mathcal{K}(\mathbb{R}^m)$.  As usual, a domain in $\mathbb{R}^m$ means a connected open subset of $\mathbb{R}^m$.
For $r>0$ and $p\in \mathbb{R}^m$ let $B^{m}(p, r)$ be the open ball centered at $p$ of radius $r$ in $\mathbb{R}^m$, and
$B^{m}(r):=B^{m}(0, r)$, $B^{m}:=B^{m}(1)$.
We always use $J$ to denote  standard  complex structure on $\mathbb{R}^{2n}$, $\mathbb{R}^{2n-2} $ and $\mathbb{R}^{2} $ without confusions.
With the linear coordinates $(q_1,\cdots,q_n,p_1,\cdots,p_n)$ on
 $\mathbb{R}^{2n}$ it  is given by the matrix
$$
J=\left(
           \begin{array}{cc}
             0 & -I_n \\
             I_n & 0 \\
           \end{array}
         \right)$$
where $I_n$ denotes the identity matrix of order $n$.
We also use ${\rm GL}(n)$ and ${\rm O}(n)$ to denote the set of invertible real matrix and
 orthogonal real matrix of order $n$, respectively.

 For a convex body $K\subset \mathbb{R}^{2n}$  containing $0$ in its interior, let
 \begin{equation}\label{e:Minkowski}
j_K:\mathbb{R}^{2n}\to\mathbb{R},\quad  j_K(z)=\inf\left\{\lambda>0\;\Big|\; \frac{z}{\lambda}\in K\right\}
 \end{equation}
 be the Minkowski functional of $K$ and let
 $$
 h_K:\mathbb{R}^{2n}\to\mathbb{R},\quad h_K(z)=\sup\{\langle x,z\rangle\,|\,x\in K\}
 $$
 be the  support function of $K$. The polar body of $K$  is defined by  $K^{\circ}=\{x\in\mathbb{R}^{2n}\,|\, \langle x,y\rangle\le 1\;\forall y\in K\}$. Then $h_K=j_{K^{\circ}}$ (\cite[Theorem~1.7.6]{Sch93}).
For two convex bodies $D, K\subset \mathbb{R}^{2n}$ containing  $0$ in their interiors
 and a real number $p\ge 1$,  there exists a unique {convex body}
 $D+_pK\subset\mathbb{R}^{2n}$ with support function
$$\mathbb{R}^{2n}\ni w\mapsto h_{D+_pK}(w)=(h^p_{D}(w)+h_{K}^p(w))^{\frac{1}{p}}$$
  (\cite[Theorem~1.7.1]{Sch93}). $D+_pK$ is called
   the $p$-sum of $D$ and $K$ by Firey (cf. \cite[(6.8.2)]{Sch93}).

  For any two convex bodies  $D, K\subset \mathbb{R}^{2n}$  containing  $0$ in their interiors,
  Artstein-Avidan and Ostrover  \cite{AAO08} proved that their Ekeland-Hofer-Zehnder
symplectic capacities satisfy the following Brunn-Minkowski type inequality
     \begin{equation}\label{e:BrunA0}
   \left(c_{\rm EHZ}(D+_pK)\right)^{\frac{p}{2}}\ge \left(c_{\rm EHZ}(D)\right)^{\frac{p}{2}}+ \left(c_{\rm EHZ}(K)\right)^{\frac{p}{2}},
   \quad  p\in\mathbb{R}\;\&\; p\ge 1.
   \end{equation}
As applications, Artstein-Avidan and Ostrover \cite{AAO14} used them  to derive several very interesting bounds and inequalities
for the length of the shortest periodic billiard trajectory in a smooth convex body in $\mathbb{R}^n$.

Recently, we established extended versions of  Ekeland-Hofer and Hofer-Zehnder
symplectic capacities in \cite{JinLu}
\footnote{The preprint was split into two papers, which were submitted independently. The present paper is one of them, mainly consisting of
contents in Sections~8,~9 of \cite{JinLu}.}, which are not symplectic capacities in general.
For the reader's convenience, we recall the definition of the extended Hofer-Zehnder symplectic capacities with respect to symplectomorphisms on symplectic manifolds  (Definition~\ref{def: ehz}) and also some related properties in Section~\ref{sec:capacity}.
In particular, for given $\Psi\in{\rm Sp}(2n,\mathbb{R})$ and $B\subset\mathbb{R}^{2n}$ such that $B\cap {\rm Fix}(\Psi)\neq \emptyset$,
 we  constructed the extended versions of  Ekeland-Hofer capacity $c_{\rm EH}(B)$ and Hofer-Zehnder capacity $c_{\rm HZ}(B)$
 relative to $\Psi$,  denoted respectively by
 $$
 c^{\Psi}_{\rm EH}(B) \quad \hbox{and} \quad c^{\Psi}_{\rm HZ}(B).
 $$
 If $\Psi=I_{2n}$, then $c^{\Psi}_{\rm EH}(B)=c_{\rm EH}(B)$  and $c^{\Psi}_{\rm HZ}(B)=c_{\rm HZ}(B)$.
 As the Ekeland-Hofer and Hofer-Zehnder symplectic capacities,
 $c^{\Psi}_{\rm EH}$  and $c^{\Psi}_{\rm HZ}$ agree on any convex body $D\subset\mathbb{R}^{2n}$.
 In this case we denote
 $$
 c^\Psi_{\rm EHZ}(D):=c^\Psi_{\rm HZ}(D,\omega_0)(=c^\Psi_{\rm EH}(D))
 $$
 and refer to it as extended Ekeland-Hofer-Zehnder capacity of $D$.
Because of these, it is natural to  generalize work by Artstein-Avidan and Ostrover  \cite{AAO08} and \cite{AAO14}.
The precise versions will be stated in the following two subsections, respectively.

\subsection{A Brunn-Minkowski type inequality for $c^\Psi_{\rm EHZ}$-capacity of convex bodies}\label{sec:1.AAO}

%

Here is the first main result of this paper.

\begin{thm}\label{th:Brun}
   Let $D, K\subset \mathbb{R}^{2n}$ be two convex bodies  containing  $0$ in their interiors. Then for any $\Psi\in{\rm Sp}(2n,\mathbb{R})$ and any real $p\ge 1$ it holds that
   \begin{equation}\label{e:BrunA}
   \left(c^{\Psi}_{\rm EHZ}(D+_pK)\right)^{\frac{p}{2}}\ge \left(c^{\Psi}_{\rm EHZ}(D)\right)^{\frac{p}{2}}+ \left(c^{\Psi}_{\rm EHZ}(K)\right)^{\frac{p}{2}}.
   \end{equation}
   Moreover,  the equality in (\ref{e:BrunA}) holds if  $D$ and $K$ satisfy the condition:
   \begin{equation}\label{e:BrunAA}
 \left.\begin{array}{ll}
 &\hbox{There exist $c^\Psi_{\rm EHZ}$-carriers for $D$ and $K$,
      $\gamma_D:[0,T]\rightarrow \partial D$ and}\\
     & \hbox{$\gamma_K:[0,T]\rightarrow \partial K$, such that
      they coincide up to  dilation and}\\
      &\hbox{translation by elements in ${\rm Ker}(\Psi-I_{2n})$, i.e.,
   $\gamma_D=\alpha\gamma_K+ {\bf b}$}\\
  &\hbox{ for some $\alpha\in\mathbb{R}\setminus\{0\}$ and ${\bf b}\in{\rm Ker}(\Psi-I_{2n})\subset\mathbb{R}^{2n}$.}
   \end{array}\right\}
      \end{equation}
When $p>1$ the condition (\ref{e:BrunAA}) is also necessary for the equality in (\ref{e:BrunA}) holding.
  \end{thm}
Readers can refer to Definition~\ref{def: carrier} for the concept of $c^\Psi_{\rm EHZ}$-carriers for a convex body.
Theorem~\ref{th:Brun} has some interesting corollaries, see Section~\ref{sec:Brunn.2}.

\subsection{Length estimate for a class of non-periodic billiard trajectories in convex domains}\label{sec:1.3}

Using the inequality (\ref{e:BrunA0}) and its corollaries Artstein-Avidan and Ostrover \cite{AAO14}
studied  the length estimates of the shortest periodic billiard trajectory in a smooth convex body in $\mathbb{R}^n$
and obtained some very interesting results. Since the Ekeland-Hofer capacity of a smooth convex body
$D\subset\mathbb{R}^{2n}$
is equal to the minimum of absolute values of actions of closed characteristics on the boundary $\partial D$,
and we generalized this relation to our extended  Ekeland-Hofer-Zehnder capacity $c^\Psi_{\rm EHZ}(D)$ and
$\Psi$-characteristics on $\partial D$ in \cite{JinLu},  it is natural using Theorem~\ref{th:Brun} or Corollaries~\ref{cor:Brun.1}, \ref{cor:Brun.2}
to study corresponding conclusions for some non-periodic billiard trajectory in a smooth convex body in $\mathbb{R}^n$,
which motivates the following definitions.



\begin{definition}\label{def:billT.1}
{\rm For a convex body $\Omega\subset\mathbb{R}^n$ with boundary $\partial\Omega$
of class $C^2$ and $A\in{\rm O}(n)$,
a nonconstant, continuous, and piecewise $C^\infty$ path $\sigma:[0,T]\to\overline{\Omega}$ with $\sigma(T)=A\sigma(0)$
is called an \textsf{$A$-billiard trajectory} in  $\Omega$ if there exists a finite set
$\mathscr{B}_\sigma\subset (0, T)$
such that $\ddot{\sigma}\equiv 0$ on $(0, T)\setminus\mathscr{B}_\sigma$ and the following conditions
are also satisfied:
\begin{description}
\item[(ABi)] $\sharp\mathscr{B}_\sigma\ge 1 $  and $\sigma(t)\in\partial\Omega\;\forall t\in \mathscr{B}_\sigma$.
\item[(ABii)] For each $t\in \mathscr{B}_\sigma$, $\dot{\sigma}^\pm(t):=\lim_{\tau\to t\pm}\dot{\sigma}(\tau)$
fulfils the equation
\begin{equation}\label{e:bill1}
\dot{\sigma}^+(t)+\dot{\sigma}^-(t)\in T_{\sigma(t)}\partial\Omega,\quad
\dot{\sigma}^+(t)-\dot{\sigma}^-(t)\in (T_{\sigma(t)}\partial\Omega)^\bot\setminus\{0\}.
\end{equation}
(So
$|\dot{\sigma}^+(t)|^2-|\dot{\sigma}^-(t)|^2=\langle\dot{\sigma}^+(t)+\dot{\sigma}^-(t), \dot{\sigma}^+(t)-\dot{\sigma}^-(t)\rangle_{\mathbb{R}^n}=0$ for each $t\in \mathscr{B}_\sigma$,
that is, $|\dot{\sigma}|$ is constant  on $(0, T)\setminus\mathscr{B}_\sigma$.)
Let
\begin{equation}\label{e:bill2}
\dot{\sigma}^+(0)=\lim_{t\to0+}\dot{\sigma}(t)\quad\hbox{and}\quad
\dot{\sigma}^-(T)=\lim_{t\to T-}\dot{\sigma}(t).
\end{equation}
If $\sigma(0)\in\partial\Omega$ (resp. $\sigma(T)\in\partial\Omega$)
let $\dot{{\sigma}}^-(0)$ (resp. $\dot{{\sigma}}^+(T)$) be the unique vector satisfying
 \begin{equation}\label{e:bill3}
\dot{\sigma}^+(0)+\dot{{\sigma}}^-(0)\in T_{\sigma(0)}\partial\Omega,\quad
\dot{\sigma}^+(0)-\dot{{\sigma}}^-(0)\in (T_{\sigma(0)}\partial\Omega)^\bot
\end{equation}
(resp.
\begin{equation}\label{e:bill4}
\dot{{\sigma}}^+(T)+\dot{\sigma}^-(T)\in T_{\sigma(T)}\partial\Omega,\quad
\dot{{\sigma}}^+(T)-\dot{\sigma}^-(T)\in (T_{\sigma(T)}\partial\Omega)^\bot.\quad{\rm )}
\end{equation}

\item[(ABiii)] If $\{\sigma(0), \sigma(T)\}\in {\rm int} \Omega$ then
  \begin{equation}\label{e:bill5}
A\dot{\sigma}^+(0)=\dot{\sigma}^-(T).
\end{equation}
\item[(ABiv)] If $\sigma(0)\in\partial\Omega$ and $\sigma(T)\in{\rm int}\Omega$, then
either (\ref{e:bill5}) holds, or
\begin{equation}\label{e:bill6}
A\dot{{\sigma}}^-(0)=\dot{\sigma}^-(T).
\end{equation}
\item[(ABv)]
If $\sigma(0)\in{\rm int}\Omega$ and $\sigma(T)\in\partial\Omega$, then
either (\ref{e:bill5}) holds, or
\begin{equation}\label{e:bill7}
A\dot{\sigma}^+(0)=\dot{{\sigma}}^+(T).
\end{equation}
\item[(ABvi)]
If $\{\sigma(0), \sigma(T)\}\in\partial\Omega$, then
either (\ref{e:bill5}) or (\ref{e:bill6}) or (\ref{e:bill7}) holds, or
\begin{equation}\label{e:bill8}
A\dot{{\sigma}}^-(0)=\dot{{\sigma}}^+(T).
\end{equation}
\end{description}}
\end{definition}

\begin{remark}\label{rm:billT.1}
{\rm
\begin{description}
  \item[(i)]For each $t\in \mathcal{B}_\sigma$, (\ref{e:bill1}) is a reflection condition which describes the motion of a billiard when arriving at the boundary of the billiard table.
  \item[(ii)]Roughly speaking, $A$-billiard trajectory requires a billiard  trajectory to satisfy boundary conditions for starting position and ending position, as well as for starting velocity and ending velocity. If $A=I_n$, an $A$-billiard trajectory becomes periodic (or closed). In this case,
      $\sigma(T)=\sigma(0)$ and (ABiv) and (ABv) do not occur. If (ABiii) holds then all bounce times of this periodic billiard trajectory  $\sigma$
consist of elements of $\mathscr{B}_\sigma$. If $\sigma(0)=\sigma(T)\in\partial\Omega$ and either (\ref{e:bill5}) or
(\ref{e:bill8}) holds then
the periodic billiard trajectory  $\sigma$ is tangent to
$\partial\Omega$ at $\sigma(0)$, and so the set of its bounce times is also $\mathscr{B}_\sigma$.
When $\sigma(0)=\sigma(T)\in\partial\Omega$ and either (\ref{e:bill6}) or (\ref{e:bill7}) holds,
it follows from (\ref{e:bill3})-(\ref{e:bill4}) that
$$
\dot{\sigma}^+(0)+\dot{{\sigma}}^-(T)\in T_{\sigma(0)}\partial\Omega\quad\hbox{and}\quad
\dot{\sigma}^+(0)-\dot{{\sigma}}^-(T)\in (T_{\sigma(0)}\partial\Omega)^\bot.
$$
When $\dot{\sigma}^+(0)-\dot{{\sigma}}^-(T)= 0$, the set of all bounce times of this periodic billiard trajectory  $\sigma$
is $\mathscr{B}_\sigma$. When $\dot{\sigma}^+(0)-\dot{{\sigma}}^-(T)= \neq 0$ , the set of all bounce times of this periodic billiard trajectory  $\sigma$
is $\mathscr{B}_\sigma\cup\{0\}=\mathscr{B}_\sigma\cup\{T\}$ (because
$0$ and $T$ are identified).
\item[(iii)]If $A\ne I_n$, an $A$-billiard trajectory in  $\Omega$ might not be periodic
even if $\sigma(0)=\sigma(T)$ since the starting velocity and ending velocity may not satisfy the condition for periodic billiard trajectory.
  \end{description}}
\end{remark}

The existence of  $A$-billiard trajectories in  $\Omega$ will be studied in other places.

Definition \ref{def:billT.1} can be generalized to convex domain with non-smooth boundary.
Recall that for a convex body $\Delta\in\mathbb{R}^n$ and $q\in\partial \Delta$
$$N_{\partial \Delta}(q)=\{y\in\mathbb{R}^{2n}\,|\, \langle u-q, y\rangle\le 0\;\forall u\in \Delta\}$$
    is the normal cone to $\Delta$ at $q\in\partial \Delta$. $y\in N_{\partial \Delta}(q)$  is called
an \textsf{outward support vector} of $\Delta$ at $q\in\partial\Delta$ .
It is unique if $q$ is a smooth point of $\partial\Delta$.
Corresponding to the generalized periodic billiard trajectory introduced by Ghomi \cite{Gh04}, we have
the following generalized version of the billiard trajectory in Definition~\ref{def:billT.1}.

\begin{definition}\label{def:billT.2}
{\rm
For a  convex body in $\Delta\subset\mathbb{R}^n$ and $A\in{\rm O}(n)$,
 a \textsf{generalized $A$-billiard trajectory} in  $\Delta$ is  defined to be
 a finite sequence of points in  $\Delta$
$$
q=q_0,q_1,\cdots,q_m=Aq
$$
 with the following properties:
\begin{description}
\item[(AGBi)] $m\ge 2$ and $\{q_1,\cdots,q_{m-1}\}\subset\partial \Delta$.
\item[(AGBii)] Both $q_0,\cdots,q_{m-1}$ and $q_1,\cdots,q_m$ are sequences of distinct points.
\item[(AGBiii)] For every $i=1,\cdots,m-1$,
$$
\nu_i:=\frac{q_i-q_{i-1}}{\|q_i-q_{i-1}\|}+ \frac{q_{i}-q_{i+1}}{\|q_{i}-q_{i+1}\|}
$$
is an outward support vector of $\Delta$ at $q_i$.
\item[(AGBiv)] If $\{q, Aq\}\subset{\rm int}(\Delta)$ then 
\begin{equation}\label{e:bill10}
\frac{A(q_1-q_{0})}{\|q_1-q_{0}\|}=\frac{q_m-q_{m-1}}{\|q_m-q_{m-1}\|}.
\end{equation}
\item[(AGBv)] If $q\in\partial\Delta$ and $Aq\in{\rm int}(\Delta)$, then
 either (\ref{e:bill10}) holds or there exists a unit vector $b_0\in\mathbb{R}^n$ such that
 \begin{equation}\label{e:bill11}
 \nu_0:=b_0-\frac{q_1-q_{0}}{\|q_1-q_{0}\|}\in N_{\partial \Delta}(q) \quad{\rm and}\quad Ab_0=\frac{q_m-q_{m-1}}{\|q_m-q_{m-1}\|}.
 \end{equation}

\item[(AGBvi)] If $q\in{\rm int}(\Delta)$ and $Aq\in\partial\Delta$, then 
either (\ref{e:bill10}) holds or there exists a unit vector $b_m\in\mathbb{R}^n$ such that
\begin{equation}\label{e:bill12}
\nu_m:=\frac{q_m-q_{m-1}}{\|q_m-q_{m-1}\|}-b_m\in N_{\partial \Delta}(Aq)\quad{\rm and}\quad
\frac{A(q_1-q_{0})}{\|q_1-q_{0}\|}=b_m.
\end{equation}
\item[(AGBvii)] If $\{q,Aq\}\subset\partial\Delta$, then 
either (\ref{e:bill10}) or (\ref{e:bill11}) or (\ref{e:bill12}) holds, or
there exist unit vectors $b'_0, b'_m\in\mathbb{R}^n$
such that
\begin{equation}\label{e:bill9}
\nu_0:=b'_0-\frac{q_1-q_{0}}{\|q_1-q_{0}\|}\in N_{\partial \Delta}(q),\quad
\nu_m:=\frac{q_m-q_{m-1}}{\|q_m-q_{m-1}\|}-b'_m\in N_{\partial \Delta}(Aq)
\quad\hbox{and}\quad
Ab'_0=b'_m.
\end{equation}
\end{description}}
\end{definition}

\begin{remark}\label{rm:billT.2}
{\rm
\begin{description}
  \item[(i)]It is easily checked that a generalized $I_n$-billiard trajectory in  $\Delta$
is exactly a generalized periodic billiard trajectory in the sense of \cite{Gh04}.
\item[(ii)]For a smooth convex body in $\Delta\subset\mathbb{R}^n$
and $A\in{\rm O}(n)$, a nonconstant, continuous,
and piecewise $C^\infty$ path $\sigma:[0,T]\to \Delta$ with $\sigma(T)=A\sigma(0)$
is an  $A$-billiard trajectory in  $\Delta$ with $\mathscr{B}_\sigma=\{t_1<\cdots<t_{m-1}\}$
if and only if the sequence
$$
q_0=\sigma(0), q_1=\sigma(t_1),\cdots,
q_{m-1}=\sigma(t_{m-1}), q_m=\sigma(T)
$$
is a generalized $A$-billiard trajectory in  $\Delta$.
\end{description}}
\end{remark}

In order to study $A$-billiard via extended Ekeland-Hofer-Zehnder capacity, we will define $(A, \Delta, \Lambda)$-\textsf{billiard trajectory} for
$A\in{\rm GL}(n)$ and convex domians $\Delta\subset\mathbb{R}^n_q$ and $\Lambda\subset\mathbb{R}^n_p$, following the idea in \cite{AAO14} which defines closed $(\Delta,\Lambda)$-billiard trajectory.

Suppose that
 {\bf $\Delta\subset\mathbb{R}^n_q$ and $\Lambda\subset\mathbb{R}^n_p$
 are two smooth convex bodies containing the origin in their interiors}. Then
 $\Delta\times \Lambda$ is a smooth manifold with corners $\partial \Delta\times\partial \Lambda$
 in the standard symplectic space
 $(\mathbb{R}^{2n},\omega_0)=(\mathbb{R}^n_q\times \mathbb{R}^n_p, dq\wedge dp)$.
 Note that $\partial(\Delta\times\Lambda)=(\partial\Delta\times\partial\Lambda)\cup
({\rm Int}(\Delta)\times\partial\Lambda)\cup
(\partial\Delta\times{\rm Int}(\Lambda))$.
 Since
$j_{\Delta\times \Lambda}(q,p)=\max\{j_\Delta(q), j_\Lambda(p)\}$, we have
$$
\nabla j_{\Delta\times \Lambda}(q,p)=\left\{
\begin{array}{cc}
 (0,\nabla j_\Lambda(p))  &\forall (q,p)\in{\rm Int}(\Delta)\times\partial \Lambda, \\
  (\nabla j_\Delta(q),0) &\forall (q,p)\in \partial \Delta\times{\rm Int}(\Lambda).
\end{array}
\right.
$$
Moreover, for $(q,p)\in\partial\Delta\times\partial\Lambda$ there holds
\begin{eqnarray*}
N_{\partial (\Delta\times\Lambda)}(q,p)&=&\{(y_1,y_2)\;|\;y_1\in N_{\partial\Delta}(q),\; y_2\in N_{\partial \Lambda}(p)\}\\
&=&\{\mu(\nabla j_\Delta(q),0)+\lambda (0,\nabla j_\Lambda (p))\;|\;\lambda\ge 0,\;\mu\ge 0\}.
\end{eqnarray*}
Define
$$
\mathfrak{X}(q,p):=J\nabla j_{\Delta\times \Lambda}(q,p)=\left\{
\begin{array}{cc}
 (-\nabla j_\Lambda(p),0)  &\forall (q,p)\in{\rm Int}(\Delta)\times\partial \Lambda, \\
  (0, \nabla j_\Delta(q)) &\forall (q,p)\in \partial \Delta\times{\rm Int}(\Lambda).
\end{array}
\right.
$$
It is well-known that every $A\in{\rm GL}(n)$ induces  a natural linear symplectomorphism
\begin{equation}\label{e:linesympl}
\Psi_A:\mathbb{R}^n_q\times \mathbb{R}^n_p\to \mathbb{R}^n_q\times
\mathbb{R}^n_p,\;(q, v)\mapsto (Aq, (A^t)^{-1}v),
\end{equation}
where $A^t$ is the transpose of $A$.

\begin{definition}\label{def:billT.3}
{\rm Let $A\in{\rm GL}(n)$, and let $\Delta\subset\mathbb{R}^n_q$ and $\Lambda\subset\mathbb{R}^n_p$
 be two smooth convex bodies containing the origin in their interiors. A continuous and piecewise smooth map $\gamma:[0, T]\to
\partial(\Delta\times\Lambda)$ with $\gamma(T)=\Psi_A\gamma(0)$ is called an  $(A, \Delta, \Lambda)$-\textsf{billiard trajectory} if
\begin{description}
\item[(BT1)] for some positive constant $\kappa$ it holds that $\dot{\gamma}(t)=\kappa\mathfrak{X}(\gamma(t))$ on $[0, T]\setminus\gamma^{-1}(\partial\Delta\times\partial\Lambda)$;
\item[(BT2)] $\gamma$ has
a right derivative $\dot{\gamma}^+(t)$ at any $t\in \gamma^{-1}(\partial\Delta\times\partial\Lambda)\setminus\{T\}$
 and a left derivative $\dot{\gamma}^-(t)$ at any $t\in \gamma^{-1}(\partial\Delta\times\partial\Lambda)\setminus\{0\}$,
 and  $\dot{\gamma}^\pm(t)$ belong to
    \begin{equation}\label{e:linesymp3}
   \{-\lambda(\nabla j_\Lambda(\gamma_p(t)),0)+\mu(0,\nabla j_\Delta(\gamma_q(t)))\,|\,
    \lambda\ge 0, \;\mu\ge 0,\;(\lambda,\mu)\ne (0,0)\}
    \end{equation}
    with $\gamma(t)=(\gamma_q(t),\gamma_p(t))$.
\end{description}}
\end{definition}

\begin{remark}\label{rem:ADL-BT}
{\rm
\begin{description}
  \item[(i)] 
Every $(A, \Delta, \Lambda)$-billiard trajectory is a generalized $\Psi_A$-characteristic on $\partial(\Delta\times\Lambda)$ in the sense of Definition~\ref{def:character}(ii).
 In fact, we only need to note that for $(q,p)\in \partial \Delta\times {\rm Int} (\Lambda)\cup (\rm Int)\Delta\times \partial \Lambda$ there holds
 $$\mathfrak{X}(q,p)=J\nabla j_{\Delta\times \Lambda}(q,p)$$
 and for $(q,p)\in \partial \Delta\times \partial\Lambda$ there holds
  $$JN_{\partial (\Delta\times\Lambda)}=\{-\lambda(\nabla j_\Lambda(\gamma_p(t)),0)+\mu(0,\nabla j_\Delta(\gamma_q(t)))\,|\,
    \lambda\ge 0, \;\mu\ge 0,\;(\lambda,\mu)\ne (0,0)\}.$$
    \item[(ii)]For a given $A\in{\rm GL}(n)$, we can generalize Definition~\ref{def:billT.3} to smooth convex bodies
    $\Delta\subset\mathbb{R}^n_q$ and $\Lambda\subset\mathbb{R}^n_p$  satisfying
\begin{equation}\label{e:linesymp2}
{\rm Fix}(A)\cap{\rm Int}(\Delta)\ne\emptyset\quad\hbox{and}\quad
{\rm Fix}(A^t)\cap{\rm Int}(\Lambda)\ne\emptyset,
\end{equation}
(which not necessarily contain the origin in their interiors). In this case,  a continuous and piecewise smooth map $\gamma:[0, T]\to
\partial(\Delta\times\Lambda)$ is said to be an  $(A, \Delta, \Lambda)$-\textsf{billiard trajectory} if there exists
$\bar{q}\in{\rm Fix}(A)\cap{\rm Int}(\Delta)$ and $\bar{p}\in {\rm Fix}(A^t)\cap{\rm Int}(\Lambda)$
such that $\gamma-(\bar{q},\bar{p})$ is an $(A, \Delta-\bar{q}, \Lambda-\bar{p})$-billiard trajectory in the sense of Definition~\ref{def:billT.3}. (Here $\gamma-(\bar{q},\bar{p})$
 is the composition of $\gamma$ and the affine linear symplectomorphism
\begin{equation}\label{e:affine}
\Phi_{(\bar{q},\bar{p})}:\mathbb{R}^n_q\times \mathbb{R}^n_p\to \mathbb{R}^n_q\times
\mathbb{R}^n_p,\;(u, v)\mapsto (u-\bar{q}, v-\bar{p}),
\end{equation}
which commutes with  $\Psi_A$.) The condition (\ref{e:linesymp2}) insures that
$$
{\rm Int}(\Delta\times\Lambda)\cap{\rm Fix}(\Psi_A)\neq\emptyset
$$
so that $c^{\Psi_A}_{\rm EHZ}(\Delta\times\Lambda)$ is well defined and we can associate the lengths of $(A, \Delta,\Lambda)$-billiard trajectories with it. 
\end{description}}
\end{remark}

 Corresponding to the classification for closed $(\Delta,\Lambda)$-trajectories in \cite{AAO14}
we introduce: 

\begin{definition}
 {\rm Let $A$, $\Delta$ and $\Lambda$ satisfy (\ref{e:linesymp2}).
 An  $(A, \Delta, \Lambda)$-billiard trajectory is called \textsf{proper} (resp. \textsf{gliding}) if $\gamma^{-1}(\partial\Delta\times\partial\Lambda)$ is a finite set (resp. $\gamma^{-1}(\partial\Delta\times\partial\Lambda)=[0,T]$,
i.e., $\gamma([0,T])\subset\partial\Delta\times\partial\Lambda$ completely).}
\end{definition}

For  $A\in{\rm GL}(n,\mathbb{R}^n)$ and  convex bodies
 $\Delta\subset\mathbb{R}^n_q$ and $\Lambda\subset\mathbb{R}^n_p$
  satisfying (\ref{e:linesymp2}), we define
\begin{equation}\label{e:linesymp5}
\xi^A_\Lambda(\Delta)=c^{\Psi_A}_{\rm EHZ}(\Delta\times\Lambda)\quad\hbox{and}\quad
\xi^A(\Delta)=c^{\Psi_A}_{\rm EHZ}(\Delta\times B^n).
\end{equation}
If $A=I_n$ then $\xi^A(\Delta)$ becomes  $\xi(\Delta)$ defined in \cite[page 177]{AAO14}.
Clearly, $\xi^A_{\Lambda_1}(\Delta_1)\le \xi^A_{\Lambda_2}(\Delta_2)$ if
both are well-defined and $\Lambda_1\subset\Lambda_2$ and $\Delta_1\subset\Delta_2$.

In Section~\ref{sec:BillT-1}, based on studies on the above
several classes of billiard trajectories we show in Proposition~\ref{prop:billT.6}
 that $\xi^A(\Delta)$ provides a positive lower bound
 for infimum of length of $A$-billiard trajectories in $\Delta$.
 Therefore it is important to study properties of $\xi^A(\Delta)$ and more general $\xi^A_\Lambda(\Delta)$.
As in the proof of \cite[Theorem~1.1]{AAO14} using Corollary~\ref{cor:Brun.1} we may derive the following Brunn-Minkowski type
inequality for $\xi^A_\Lambda$, which is the second main result of this paper.

\begin{thm}\label{th:billT.2}
 For $A\in{\rm GL}(n)$, suppose that convex bodies
  $\Delta_1, \Delta_2\subset\mathbb{R}^n_q$ and $\Lambda\subset\mathbb{R}^n_p$
 satisfy  ${\rm Int}(\Delta_1)\cap{\rm Fix}(A)\ne\emptyset$,
 ${\rm Int}(\Delta_2)\cap{\rm Fix}(A)\ne\emptyset$
 and ${\rm Int}(\Lambda)\cap{\rm Fix}(A^t)\ne\emptyset$.    Then
\begin{equation}\label{e:linesymp6}
 \xi^A_\Lambda(\Delta_1+\Delta_2)\ge \xi^A_\Lambda(\Delta_1)+ \xi^A_\Lambda(\Delta_2)
   \end{equation}
   and the equality holds if  there exist
    $c^{\Psi_A}_{\rm EHZ}$-carriers for $\Delta_1\times \Lambda$ and $\Delta_2\times \Lambda$ which coincide up to dilation and translation
    by elements in ${\rm Ker}(\Psi_A-I_{2n})$.
\end{thm}

 When $\Lambda=B^n$ and $A=I_{n}$,
this result was first proved in \cite{AAO14}, and Irie also gave a new proof in \cite{Ir15}.


In order to estimate $\xi^A(\Delta)$, 
 for a symplectic  matrix  $\Psi\in{\rm Sp}(2n,\mathbb{R})$ we define
  \begin{equation}\label{e:TPsi0}
      g^{\Psi}:\mathbb{R}\rightarrow \mathbb{R}, \,s\mapsto \det (\Psi-e^{sJ}),
   \end{equation}
  where  $e^{tJ}=\sum^\infty_{k=0}\frac{1}{k!}t^k J^k$.
 The set of zeros of  $g^{\Psi}$ in $(0, 2\pi]$
 is a nonempty finite set (\cite[Lemma~A.1]{JinLu}) and
 \begin{equation}\label{e:TPsi}
\mathfrak{t}(\Psi):=\min\{t\in (0, 2\pi]\,|\, g^{\Psi}(t)=0\}=2c^{\Psi}_{\rm EHZ}(B^{2n})
 \end{equation}
 by \cite[(1.28)]{JinLu}. In particular, if $\Psi=I_{2n}$ then $\mathfrak{t}(\Psi)=2\pi$ (\cite[Lemma~A.1]{JinLu}) and (\ref{e:TPsi}) becomes
 $c_{\rm EHZ}(B^{2n})=\pi$.
 Since $\Psi_A={\rm diag}(A, (A^t)^{-1})$ for $A\in GL(n)$, by \cite[Lemma~A.5]{JinLu}, $ \mathfrak{t}(\Psi_A)$ is equal to
 the smallest zero in $(0, 2\pi]$ of the function
 \begin{equation}\label{e:TPsi1}
\mathbb{R}\rightarrow \mathbb{R},\;\,s\mapsto\det(I_n+ (A^t)^{-1}A-\cos s (A+ (A^t)^{-1})).
\end{equation}
(It must exist!) Moreover, if $A$ is an orthogonal matrix similar to one of form
\cite[(A.2)]{JinLu}, i.e.,
$$
A=
{\rm diag}\left(  \left(
            \begin{array}{cc}
              \cos\theta_1 & \sin\theta_1 \\
              -\sin\theta_1 & \cos\theta_1\\
            \end{array}
          \right),\cdots,
    \left(
            \begin{array}{cc}
              \cos\theta_m & \sin\theta_m \\
              -\sin\theta_m & \cos\theta_m\\
            \end{array}
          \right), I_k, -I_l\right),
$$
where $2m+k+l=n$ and $0<\theta_1\le\cdots\le\theta_m<\pi$, then
\begin{equation}\label{e:TPsi2}
 \mathfrak{t}(\Psi_A)=\left\{
\begin{array}{lll}
 \theta_1  &\hbox{if}\;m>0, \\
 \pi &\hbox{if\; $m=0$ and $l>0$},\\
 2\pi &\hbox{if\; $m=l=0$}.
\end{array}
\right.
\end{equation}

 The \textsf{width} of a convex body $\Delta\subset\mathbb{R}^n_q$ is the thickness of the
narrowest slab which contains $\Delta$, i.e.,
${\rm width}(\Delta)=\min\{h_\Delta(u)+ h_\Delta(-u)\,|\, u\in S^n\}$,
where $S^n=\{u\in\mathbb{R}^n\;|\;\|u\|=1\}$.
Let
\begin{eqnarray}\label{e:linesymp12}
&&S^n_\Delta:=\{u\in S^n\,|\, {\rm width}(\Delta)=h_\Delta(u)+ h_\Delta(-u)\},\\
&&H_u:=\{x\in\mathbb{R}^n\,|\,\langle x, u\rangle=(h_\Delta(u)- h_\Delta(-u))/2\},\label{e:linesymp13}\\
&&Z^{2n}_\Delta:=([-{\rm width}(\Delta)/2, {\rm width}(\Delta)/2]\times\mathbb{R}^{n-1})\times ([-1,1]\times\mathbb{R}^{n-1}).\label{e:linesymp14}
\end{eqnarray}


\begin{proposition}\label{prop:add}
 Let $A\in{\rm GL}(n)$ and a convex body $\Delta\subset\mathbb{R}^n_q$
satisfy ${\rm Fix}(A)\cap{\rm Int}(\Delta)\ne\emptyset$.
\begin{description}
\item[(i)] If  $\Delta$ contains a ball $B^{n}(\bar{q},r)$ with $A\bar{q}=\bar{q}$,
then 
\begin{equation}\label{e:linesymp8}
\xi^A(\Delta)\ge rc^{\Psi_A}_{\rm EHZ}(B^n\times B^n,\omega_0)\ge \frac{r \mathfrak{t}(\Psi_A)}{2}.
\end{equation}

\item[(ii)]  For any $u\in S^n_\Delta$,
 $\bar{q}\in H_u$ and any ${\bf O}\in O(n)$ such that ${\bf O}u=e_1=(1,0,\cdots,0)\in\mathbb{R}^n$
 let
\begin{equation}\label{e:linesymp16}
\Psi_{{\bf O}, \bar{q}}:\mathbb{R}^n_q\times \mathbb{R}^n_p\to \mathbb{R}^n_q\times
\mathbb{R}^n_p,\;(q, v)\mapsto ({\bf O}(q-\bar{q}), {\bf O}v),
\end{equation}
that is, the composition of translation $(q,v)\mapsto (q-\bar{q},v)$
and $\Psi_{\bf O}$ defined by (\ref{e:linesympl}), then
\begin{eqnarray}\label{e:linesymp17}
\xi^A(\Delta)\le  c_{\rm EHZ}^{\Psi_{{\bf O}, \bar{q}}\Psi_A\Psi_{{\bf O}, \bar{q}}^{-1}}(Z^{2n}_\Delta, \omega_0).
\end{eqnarray}
Moreover,  the right-side is equal to $c_{\rm EHZ}^{\Psi_{{\bf O}}\Psi_A\Psi_{{\bf O}}^{-1}}(Z^{2n}_\Delta, \omega_0)$ if $A\bar{q}=\bar{q}$, and to
$c_{\rm EHZ}^{\Psi_A}(Z^{2n}_\Delta, \omega_0)$ if $A\bar{q}=\bar{q}$ and $A{\bf O}={\bf O}A$.
\end{description}
\end{proposition}

By Proposition~\ref{prop:billT.6} and (\ref{e:linesymp8}) we immediately
 get our third main result. 

\begin{thm}\label{th:billT.3}
For $A\in{\rm O}(n)$ and a smooth convex body $\Delta\subset\mathbb{R}^n_q$
with ${\rm Fix}(A)\cap{\rm Int}(\Delta)\ne\emptyset$, 
if  $\Delta$ contains a ball $B^{n}(\bar{q},r)$ with $A\bar{q}=\bar{q}$ then it holds that
\begin{equation}\label{e:linesymp8.1}
 \frac{r \mathfrak{t}(\Psi_A)}{2}\le \inf\{L(\sigma)\,|\,\hbox{$\sigma$ is an  $A$-billiard trajectory  in $\Delta$}\}.
\end{equation}
\end{thm}


{Recall that the \textsf{inradius} of a convex body $\Delta\subset\mathbb{R}^n_q$
 is the radius of the largest ball contained in $\Delta$, i.e.,}
 ${\rm inradius}(\Delta)=\sup_{x\in \Delta}{\rm dist}(x,\partial \Delta)$.
 For any centrally symmetric convex body $\Delta\subset\mathbb{R}^n_q$,  Artstein-Avidan,  Karasev, and Ostrover recently proved in \cite[Theorem~1.7]{AAKO14}:
 \begin{equation}\label{e:linesymp9}
 c_{\rm HZ}(\Delta\times \Delta^\circ,\omega_0)=4.
 \end{equation}
 As a consequence of this and (\ref{e:linesymp8.1}) we obtain: 

 \begin{corollary}[Ghomi \cite{Gh04}]\label{cor:billT.1}
Every periodic billiard trajectory $\sigma$ in a centrally symmetric convex body $\Delta\subset\mathbb{R}^n_q$
has length $L(\sigma)\ge 4\, {\rm inradius}(\Delta)$.
\end{corollary}
 \begin{proof}
  Since  $c^{\Psi_A}_{\rm HZ}=c_{\rm HZ}$ for $A=I_n$,  from the first inequality in (\ref{e:linesymp8}) and (\ref{e:linesymp9})  we deduce
 \begin{equation}\label{e:linesymp10}
 \xi(\Delta):=\xi^{I_n}(\Delta)\ge 4\, {\rm inradius}(\Delta).
 \end{equation}
 When $\Delta$ is smooth, since $\xi(\Delta)$ is equal to the length of the shortest
 periodic billiard trajectory  in $\Delta$ (see the bottom of \cite[page 177]{AAO14}), we get $L(\sigma)\ge 4\, {\rm inradius}(\Delta)$.
 (In this case another new proof of \cite[Theorem~1.2]{Gh04} was also given by
Irie \cite[Theorem~1.9]{Ir15}.) For general case we may approximate $\Delta$
by a smooth convex body $\Delta^\ast\supseteq\Delta$ such that $\sigma$ is
also periodic billiard trajectory $\Delta^\ast$. Thus $L(\sigma)\ge\xi(\Delta^\ast)\ge
  \xi(\Delta)\ge 4\, {\rm inradius}(\Delta)$ because of monotonicity of $c_{\rm HZ}$.
   \end{proof}

\begin{remark}\label{rem:billT.1}
{\rm
\begin{description}
 \item[(i)]
 Corollary~\ref{cor:billT.1} only partially recover \cite[Theorem~1.2]{Gh04} by Ghomi.
 \cite[Theorem~1.2]{Gh04}  did not require $\Delta$ to be centrally symmetric. It also stated that
 $L(\sigma)=4\, {\rm inradius}(\Delta)$ for some $\sigma$ if and only if
 ${\rm width}(\Delta)=4\, {\rm inradius}(\Delta)$.

 \item[(ii)]
 When $A=I_n$ we may take $r={\rm inradius}(\Delta)$ in
 (\ref{e:linesymp8.1}), and get a weaker result than Corollary~\ref{cor:billT.1}:
  $L(\sigma)\ge \pi{\rm inradius}(\Delta)$ for every periodic billiard trajectory $\sigma$ in $\Delta$.

\item[(iii)]
In order to get a corresponding result for each $A$-billiard trajectory in $\Delta$ as in Corollary~\ref{cor:billT.1},
an analogue of (\ref{e:linesymp10}) is needed. Hence we expect
that (\ref{e:linesymp9}) has the following generalization:
\begin{equation}\label{e:linesymp11}
 c^{\Psi_A}_{\rm EHZ}(\Delta\times \Delta^\circ)=\frac{2}{\pi} \mathfrak{t}(\Psi_A).
 \end{equation}
 \end{description}}
\end{remark}

For a bounded domain $\Omega\subset\mathbb{R}^n$ with smooth boundary,
there exist positive constants $C_n$, $C_n'$ only depending on $n$, $C$ independent of $n$, and
(possibly different) periodic billiard trajectories $\gamma_1$, $\gamma_2$, $\gamma_3$  in $\Omega$ such that their length satistfies
\begin{eqnarray}\label{e:linesymp18}
&&L(\gamma_1)\le C_n{\rm Vol}(\Omega)^{\frac{1}{n}}\quad\hbox{(Viterbo \cite{Vi00})},\\
&&L(\gamma_2)\le C{\rm diam}(\Omega)\quad\hbox{(Albers and Mazzucchelli \cite{AlMaz10})},\label{e:linesymp19}\\
&&L(\gamma_3)\le C_n'{\rm inradius}(\Omega)\quad\hbox{(Irie \cite{Ir12})},\label{e:linesymp20}
\end{eqnarray}
where ${\rm inradius}(\Omega)$ is the inradius of $\Omega$, i.e.,
 the radius of the largest ball contained in $\Omega$.
 If $\Omega$ is a smooth convex body $\Delta\subset\mathbb{R}^n_q$, Artstein-Avidan and Ostrover \cite{AAO14}
recently obtained the following more concrete  estimates than (\ref{e:linesymp20}) and (\ref{e:linesymp18}):
 \begin{eqnarray}\label{e:linesymp21}
&&\xi(\Delta)\le 2(n+1){\rm inradius}(\Delta),\\
&&\xi(\Delta)\le C'\sqrt{n}{\rm Vol}(\Delta)^{\frac{1}{n}},\label{e:linesymp22}
\end{eqnarray}
where $C'$ is a positive constant independent of $n$.

\begin{remark}\label{rem:billT.2}
{\rm
 Since $c^{\Psi_A}_{\rm HZ}=c_{\rm HZ}$ for $A=I_n$, from (\ref{e:linesymp17}) we  recover
(\ref{e:linesymp21}) as follows
\begin{eqnarray*}
\xi(\Delta)=\xi^{I_n}(\Delta)\le c_{\rm HZ}(Z^{2n}_\Delta, \omega_0)
=2{\rm width}(\Delta)\le 2(n+1){\rm inradius}(\Delta)
\end{eqnarray*}
because ${\rm width}(\Delta)\le (n+1){\rm inradius}(\Delta)$
by \cite[(1.2)]{Sch09}. 
}
 \end{remark}

Finally, we have an improvement for (\ref{e:linesymp19}) in the case that $\Omega$ is a smooth convex body.

\begin{theorem}\label{th:billT.4}
For a smooth convex body $\Delta\subset\mathbb{R}^n_q$,
suppose that periodic billiard trajectories in $\Delta$ include  projections to $\Delta$
  of periodic gliding billiard trajectories in $\Delta\times B^n$. Then
 $$
 L(\sigma)\le \pi{\rm diam}(\Delta)
 $$
for some periodic billiard trajectory $\sigma$ in $\Delta$.
\end{theorem}

\noindent{\bf Organization of the paper}. 
  Section~\ref{sec:Brunn} proves Theorem~\ref{th:Brun} and Corollaries~\ref{cor:Brun.1}, \ref{cor:Brun.2}.
In Section~\ref{sec:BillT-1} we give the classification of $(A, \Delta, \Lambda)$-billiard trajectories  and
studied related properties of proper trajectories.
Theorems~\ref{th:billT.2},~\ref{th:billT.4} and Proposition~\ref{prop:add} will be proved In Section~\ref{sec:BillT-2}.



\noindent{\bf Acknowledgments}. We are deeply grateful to the
anonymous referees for giving very helpful comments and suggestions to improve the
exposition.


%
%

%

\section{The extended Hofer-Zehnder symplectic capacities}\label{sec:capacity}
\setcounter{equation}{0}


For convenience we review the extended Hofer-Zehnder symplectic capacities and related results in \cite{JinLu}.
Given a symplectic manifold $(M,\omega)$ and a symplectomorphism $\Psi\in{\rm Symp}(M,\omega)$, let $O\subset M$ be an open subset such that
$O\cap {\rm Fix}(\Psi)\neq\emptyset$. Denote by $\mathcal{H}^\Psi(O,\omega)$ the set of smooth functions $H \colon O\to\R$ satisfying
\begin{description}
\item[(i)]there exists a nonempty open subset $U\subset O$ (depending on $H$) such that $U\cap{\rm Fix}(\Psi)\neq\emptyset$ and $H|_U=0$,
\item[(ii)]there exists a compact subset $K\subset O\setminus\partial O$ (depending on $H$) such that $H|_{O\setminus K}=m(H):=\max H$,
\item[(iii)]  $0\leq H\leq m(H)$.
\end{description}
Denote by $X_H$ the Hamiltonian vector field defined by {$\omega(X_H, \cdot)=-dH$}.
Note that for $H\in\mathcal{H}^\Psi(O,\omega)$, the condition $U\cap {\rm Fix}(\Psi)\neq \emptyset$  ensures that there exists a constant solution
to the Hamiltonian boundary value problem
\begin{equation}\label{bvp}
   \left\{
   \begin{array}{l}
     \dot{x}=X_H(x), \\
     x(T)=\Psi x(0).
   \end{array}
   \right.
\end{equation}
We call $H\in\mathcal{H}^\Psi(O,\omega)$
$\Psi$-\textsf{admissible} if all solutions $x:[0, T]\to O$ to the Hamiltonian boundary value problem (\ref{bvp})
with $0<T\le 1$ are constant.
The set of all such $\Psi$-admissible Hamiltonians is denoted by
 $\mathcal{H}_{ad}^{\Psi}(O,\omega)$. In \cite{JinLu} we defined the following analogue (or extended version) of the Hofer-Zehnder capacity of $(O, \omega)$.
 \begin{definition}\label{def: ehz}
 For open subset $O$ in symplectic manifold $(M,\omega)$ and symplectomorphism $\Psi\in{\rm Symp}(M,\omega)$, define
$$\displaystyle{c^\Psi_{\rm HZ}}(O,\omega)=\sup \{\max H\,|\, H\in \mathcal{H}_{ad}^{\Psi}(O,\omega)\}.$$
\end{definition}
Clearly If $\Psi=id_M$ then $c^\Psi_{\rm HZ}(O,\omega)=c_{\rm HZ}(O,\omega)$ for any open subset $O\subset M$, where $c_{\rm HZ}(O,\omega)$ is the Hofer-Zehnder capacity defined in \cite{HoZe90}.

The following proposition lists some basic properties of the extended Hofer-Zehnder capacity. In this paper, the standard symplectic structure on $\mathbb{R}^{2n}$ is given by $\omega_0=\sum_{i=1}^{n}dq_i\wedge dp_i$ with linear coordinates $(q_1,\cdots,q_n,p_1,\cdots,p_n)$.  Let ${\rm Sp}(2n,\mathbb{R})$ denote the set of symplectic matrix of order $2n$. Each symplectic matrix $\Psi\in{\rm Sp}(2n,\mathbb{R})$
is identified with the linear symplectomorphism on $(\mathbb{R}^{2n},\omega_0)$
which has the representing matrix $\Psi$ under the standard symplectic basis
of $(\mathbb{R}^{2n},\omega_0)$, $(e_1,\cdots,e_n,f_1,\cdots,f_n)$,
where the $i$-th(resp. $i+n$-th) coordinate of $e_i$ (resp. $f_{n+i}$) is $1$ and other coordinates
are zero.
\begin{proposition}[\hbox{\cite[Proposition~1.2]{JinLu}}]\label{MonComf}
\begin{description}
  \item[(i)] {\rm (Conformality.)} $c^\Psi_{\rm HZ}(M,\alpha\omega)=\alpha c^\Psi_{\rm HZ}(M,\omega)$ for any $\alpha\in\mathbb{R}_{>0}$, and
    $c^{\Psi^{-1}}_{\rm HZ}(M,\alpha\omega)=-\alpha c^\Psi_{\rm HZ}(M,\omega)$ for any $\alpha\in\mathbb{R}_{<0}$.

\item[(ii)] {\rm (Monotonicity.)}
Suppose that $\Psi_i\in{\rm Symp}(M_i,\omega_i)$ $(i=1, 2)$.
If there exists a symplectic embedding
$\phi:(M_1,\omega_1)\to (M_2,\omega_2)$ of codimension zero such that $\phi\circ\Psi_1=\Psi_2\circ\phi$,
then for open subsets $O_i\subset M_i$ with $O_i\cap{\rm Fix}(\Psi_i)\ne\emptyset$ $(i=1, 2)$ and $\phi(O_1)\subset O_2$,
it holds that $c^{\Psi_1}_{\rm HZ}(O_1,\omega_1)\le c^{\Psi_2}_{\rm HZ}(O_2,\omega_2)$.

\item[(iii)] {\rm (Inner regularity.)}   For any precompact open subset
$O\subset M$ with $O\cap{\rm Fix}(\Psi)\ne\emptyset$, we have
$$c^\Psi_{\rm HZ}(O,\omega)=\sup\{c^\Psi_{\rm HZ}(K,\omega)\,|\, K\;\hbox{open},\;K\cap{\rm Fix}(\Psi) \ne\emptyset,\;
\overline{K}\subset O\}.$$

\item[(iv)] {\rm (Continuity.)}
 For a bounded convex domain $A\subset\mathbb{R}^{2n}$, suppose that
$\Psi\in{\rm Sp}(2n, \mathbb{R})$ satisfies $A\cap{\rm Fix}(\Psi)\ne\emptyset$. Then for every $\varepsilon>0$ there exists some $\delta>0$ such that
 for all bounded convex domain $O\subset\mathbb{R}^{2n}$ intersecting with ${\rm Fix}(\Psi)$, it holds that
 $$|c^\Psi_{\rm HZ}(O,\omega_0)-c^\Psi_{\rm HZ}(A,\omega_0)|\le
\varepsilon$$
provided that $A$ and $O$ have  the Hausdorff distance $d_{\rm H}(A,O)<\delta$.
\end{description}
 \end{proposition}

\begin{remark}\label{rm:2.2}
{\rm
  \begin{description}
    \item[(i)] The two symplectomorphisms $\Psi_i\in{\rm Symp}(M_i,\omega_1)$ ($i=1, 2$) involved in the above monotonicity property are different in general.
    \item[(ii)]  By the above mononicity property, for any $\Psi, \phi\in{\rm Symp}(M,\omega)$ and
 any open subset $O\subset M$ with $O\cap{\rm Fix}(\Psi)\ne\emptyset$, there holds
\begin{equation}\label{e:inv}
c^\Psi_{\rm HZ}(O,\omega)=c_{\rm HZ}^{\phi\circ\Psi\circ\phi^{-1}}(\phi(O),\omega).
\end{equation}
In particular, denote ${\rm Symp}_{\Psi}(M,\omega):=\{\phi\in {\rm Symp}(M,\omega)\,|\,\phi\circ\Psi=\Psi\circ\phi\}$, i.e., the set of
stabilizers at $\Psi$ for the adjoint action on ${\rm Symp}(M,\omega)$. Then for any $\phi \in {\rm Symp}_{\Psi}(M,\omega)$ there holds
$$
c^\Psi_{\rm HZ}(O,\omega)=c_{\rm HZ}^{\Psi}(\phi(O),\omega).
$$
That is to say, unlike the Hofer-Zehnder capacity which is invariant under the action of ${\rm Symp}(M,\omega)$, the extended Hofer-Zehnder capacity $c^\Psi_{\rm HZ}(O,\omega)$ is only invariant under the action of a subgroup of ${\rm Symp}(M,\omega)$ related to $\Psi$.
\item[(iii)]For $\Psi\in{\rm Sp}(2n, \mathbb{R})$ and any open set
 $O\ni 0$ in $(\mathbb{R}^{2n},\omega_0)$, (i)-(ii) of Proposition~\ref{MonComf} implies
\begin{equation}\label{e:inv.1}
c^\Psi_{\rm HZ}(\alpha O,\omega_0)=\alpha^2 c_{\rm HZ}^{\Psi}(O,\omega_0),\quad\forall\alpha\ge 0.
\end{equation}
  \end{description}}
\end{remark}


In \cite{AAO08}, a key for the proof of  the inequality (\ref{e:BrunA0})
is the representation theorem for Ekeland-Hofer and Hofer-Zehnder capacity of convex bodies ({\cite{HoZe90}, \cite{EH89, EH90, Sik90}}).
To present such a representation theorem for $\displaystyle c^\Psi_{\rm EHZ}(D)$ given in \cite{JinLu},
 which is crucial for the proof of Theorem~\ref{th:Brun},
 we recall the concept of characteristic on hypersurfaces in symplectic manifolds.

\begin{definition}[\hbox{\cite[Definition~1.1]{JinLu}}]\label{def:character}
{\rm {\bf (i)} For a smooth hypersurface $\mathcal{S}$ in a symplectic manifold $(M, \omega)$
and $\Psi\in{\rm Symp}(M, \omega)$, a $C^1$ embedding $z$ from $[0,T]$ {\rm (for some $T>0$)}
 into $\mathcal{S}$ is called a $\Psi$-\textsf{characteristic} on $\mathcal{S}$
 if
 $$z(T)=\Psi z(0)\; {\rm and} \;\dot{z}(t)\in(\mathcal{L}_\mathcal{S})_{z(t)}\;\forall t\in [0,T],$$
 where $\mathcal{L}_\mathcal{S}$
 is  the characteristic line bundle  given by
 $$\mathcal{L}_\mathcal{S}={\Big\{}(x,\xi)\in T\mathcal{S}\ {\Big |}\
 {\omega}_x(\xi,\eta)=0\;\hbox{for all}\;
\eta\in T_{x}\mathcal{S}{\Big \}}.
$$
 Clearly, $z(T-\cdot)$ is a $\Psi^{-1}$-characteristic, and for any $\tau>0$ the embedding
$[0, \tau T]\to \mathcal{S},\;t\mapsto z(t/\tau)$ is also a
$\Psi$-characteristic.

\noindent{\bf (ii)} If $\mathcal{S}$ is the boundary of a convex body $D$ in $(\mathbb{R}^{2n},\omega_0)$, corresponding to the definition of closed characteristics on $\mathcal{S}$ in Definition~1 of \cite[Chap.V,\S1]{Ek90} we say
a nonconstant  absolutely continuous curve $z:[0,T]\to \mathcal{S}$ (for some $T>0$)
  to be a \textsf{generalized characteristic} on $\mathcal{S}$
  if
    $$\dot{z}(t)\in JN_\mathcal{S}(z(t))\;\hbox{a.e.},$$
     where
    $$N_\mathcal{S}(x)=\{y\in\mathbb{R}^{2n}\,|\, \langle u-x, y\rangle\le 0\;\forall u\in D\}$$
    is the normal cone to $D$ at $x\in\mathcal{S}$. If $z$ satisfies $z(T)=\Psi z(0)$  for $\Psi\in{\rm Sp}(2n,\mathbb{R})$
    in addition, then we call $z$ a \textsf{generalized $\Psi$-characteristic} on $\mathcal{S}$. For a generalized characteristic $z:[0,T]\to\mathcal{S}$,
    define its action by
    \begin{equation}\label{e:action1}
A(x)=\frac{1}{2}\int_0^T\langle -J\dot{x},x\rangle dt,
\end{equation}
where $\langle\cdot,\cdot\rangle=\omega_0(\cdot,J\cdot)$ is the standard inner product on $\mathbb{R}^{2n}$.
    }
\end{definition}
\begin{remark}
If $\mathcal{S}$ in (ii) is also  $C^{1,1}$  then
generalized $\Psi$-characteristics on $\mathcal{S}$ are $\Psi$-characteristics
up to reparameterization.
\end{remark}

As a generalization of the representation theorem for Ekeland-Hofer and Hofer-Zehnder capacity of convex bodies ({\cite{HoZe90}, \cite{EH89, EH90, Sik90}}), we have: 

\begin{thm}[\hbox{\cite[Theorem~1.8]{JinLu}}]\label{th:convex}
Let $\Psi\in{\rm Sp}(2n,\mathbb{R})$ and let $D\subset \mathbb{R}^{2n}$ be a convex bounded domain with
boundary $\mathcal{S}=\partial D$ and contain a fixed point $p$ of $\Psi$. Then
there is a generalized $\Psi$-characteristic $x^{\ast}$ on $\mathcal{S}$ such that
\begin{eqnarray}\label{e:action2}
A(x^{\ast})&=&\min\{A(x)>0\,|\,x\;\text{is a generalized}\;\Psi\hbox{-characteristic on}\;\mathcal{S}\}\\
&=&c^\Psi_{\rm EHZ}(D,\omega_0).\label{e:action-capacity1}
\end{eqnarray}
If $\mathcal{S}$ is of class $C^{1,1}$, (\ref{e:action2}) and (\ref{e:action-capacity1}) become
  $$c^\Psi_{\rm EHZ}(D,\omega_0)=A(x^{\ast})=\inf\{A(x)>0\,|\,x\;\text{is a}\;\Psi\hbox{-characteristic on}\;\mathcal{S}\}.$$
   \end{thm}
  \begin{definition}\label{def: carrier}
A generalized $\Psi$-characteristic $x^{\ast}$ on $\mathcal{S}$  satisfying
(\ref{e:action2})--(\ref{e:action-capacity1}) is called a \textsf{$\displaystyle c^\Psi_{\rm EHZ}$-carrier} for $D$.
\end{definition}



\section{Proofs of Theorem~\ref{th:Brun} and  Corollaries}\label{sec:Brunn}
\setcounter{equation}{0}

\subsection{Proof of Theorem~\ref{th:Brun}}\label{sec:Brunn.1}

The basic proof ideas are similar to those of \cite{AAO08}.
For $\Psi\in{\rm Sp}(2n)$, let $E_1\subset \mathbb{R}^{2n}$ be the eigenvector space which belongs to eigenvalue ~$1$ of ~$\Psi$ and $E_1^{\bot}$ be the orthogonal complement of $E_1$ with respect to the standard Euclidean inner product in $\mathbb{R}^{2n}$. For $p>1$, let
  $$
\mathcal{F}_p=\{x\in W^{1,p}([0,1],\mathbb{R}^{2n})\,|\,x(1)=\Psi x(0)\; \& \; x(0)\in E_1^{\bot}\},
$$
which is a subspace of $W^{1,p}([0,1],\mathbb{R}^{2n})$. Since the functional
$$
\mathcal{F}_p\ni x\mapsto  A(x)=\frac{1}{2}\int_0^1\langle -J\dot{x}(t),x(t)\rangle dt
$$
is $C^1$ and $dA(x)[x]=2$ for any $x\in\mathcal{F}_p$ with $A(x)=1$,  we deduce  that
$$
\mathcal{A}_p:=\{x\in \mathcal{F}_p\,|\,A(x)=1 \}
$$
is a regular $C^1$ submanifold.

Recall that for convex body $D\subset\mathbb{R}^{2n}$, $h_D$ is the support function (see the beginning in Section~\ref{sec:1.AAO}). If $D$ contains $0$
in its interior, then $j_D$ is the associated Minkowski function. $H_D^\ast$ is the Legendre transform of $H_D:=(j_D)^2$ .

 \begin{remark}\label{rm:3.1}
 {\rm
 \begin{description}
   \item[(i)] By the homogeneity of $H_D$ and $H_D^\ast$, there exist constants $R_1, R_2\geq 1$ such that
\begin{equation}\label{e:dualestimate0}
 \frac{|z|^2}{R_1}\leq H_D(z)\leq R_1|z|^2, \quad \frac{|z|^2}{R_2}\leq H^{\ast}_D(z)\leq R_2|z|^2, \quad
\forall z\in\mathbb{R}^{2n}.
\end{equation}
   \item[(ii)]For $p>1$, let $q=p/p-1$, denote by $\left(j_D^p/p\right)^{\ast}$ the Legendre transform of $j_D^p/p$. Then there holds
   \begin{equation}\label{e:Brunn.0}
   \left(\frac{1}{p}j_D^p\right)^{\ast}(w)=\frac{1}{q}(h_D(w))^q.
   \end{equation}
   In particular, we obtain that  $H_D^{\ast}$  and
the support function $h_D$  have the following relation:
\begin{eqnarray}\label{e:Brunn.1}
  H_D^{\ast}(w)=\frac{h_D(w)^2}{4}.
\end{eqnarray}
In fact, we can compute directly as follows:
\begin{eqnarray*}
  \left(\frac{1}{p}j_D^p\right)^{\ast}(w)&=&\sup_{\xi\in\mathbb{R}^{2n}}\bigl(\langle\xi,w\rangle-
  \frac{1}{p}(j_D^p(\xi))\bigr)\nonumber\\
             &=&\sup_{t\ge 0,\zeta\in \partial D}(\langle t\zeta,w\rangle- \frac{t^p}{p}(j_D^p(\zeta))\bigr)\nonumber\\
             &=&\sup_{\zeta\in \partial D,\langle\zeta,w\rangle\ge 0}\max_{t\ge 0}\bigl(\langle t\zeta,w\rangle-\frac{t^p}{p}\bigr)\nonumber\\
             &=&\sup_{\zeta\in \partial D,\langle\zeta,w\rangle\ge 0}\frac{\langle \zeta,w\rangle^q}{q}\nonumber\\
              &=&\sup_{\zeta\in D,\langle\zeta,w\rangle\ge 0}\frac{\langle \zeta,w\rangle^q}{q}\nonumber\\
             &=&\frac{1}{q}(h_D(w))^q.
\end{eqnarray*}
 \end{description}}
 \end{remark}

To prove Theorem~\ref{th:Brun}, we need the following representation for $(c^{\Psi}_{\rm EHZ}(D))^{\frac{p}{2}}$
for convex body $D\subset\mathbb{R}^{2n}$ and $p\ge 1$, which is a generalization of \cite[Proposition~2.1]{AAO08}.
\begin{prop}\label{prop:Brun.2}
  For $p_1>1$ and $p_2\ge 1$, there holds
  $$
  (c^{\Psi}_{\rm EHZ}(D))^{\frac{p_2}{2}}=\min_{x\in\mathcal{A}_{p_1}}
\int_0^1(H_D^{\ast}(-J\dot{x}(t)))^{\frac{p_2}{2}}dt=\min_{x\in\mathcal{A}_{p_1}}\frac{1}{2^{p_2}}\int_0^1
  (h_{D}(-J\dot{x}))^{p_2}dt.
  $$
\end{prop}

Proposition~\ref{prop:Brun.2} is derived based on the following Lemma. For the case $\Psi=I_{2n}$, it is
proved in \cite[Proposition~2.2]{AAO08}.

\begin{lemma}\label{prop:Brun.1}
For $p>1$, there holds
\begin{equation}\label{eq: p-rep}
(c^{\Psi}_{\rm EHZ}(D))^{\frac{p}{2}}=\min_{x\in\mathcal{A}_p}
\int_0^1(H_D^{\ast}(-J\dot{x}(t)))^{\frac{p}{2}}dt.
\end{equation}
\end{lemma}

We firstly give the proof of Lemma~\ref{prop:Brun.1} and Proposition~\ref{prop:Brun.2}. The proof of  Theorem~\ref{th:Brun} is given in the final part of this section.

\begin{proof}[\it Proof of Lemma~\ref{prop:Brun.1}]

Define
$$
I_p:\mathcal{F}_p\to\mathbb{R},\;x\mapsto \int_0^1(H_D^{\ast}(-J\dot{x}(t)))^{\frac{p}{2}}dt.
$$
Then $I_p$  is convex. If $D$ is strictly convex with $C^1$-smooth boundary
then $I_p$ is a $C^{1}$ functional  with derivative given by
 $$
 dI_p(x)[y]=\int_0^1\langle \nabla (H_D^{\ast})^{\frac{p}{2}}(-J\dot{x}(t)),-J\dot{y}\rangle dt,\quad\forall x,y\in \mathcal{F}_p.
 $$

By Theorem~\ref{th:convex}, in order to prove (\ref{eq: p-rep}) we only need to show that
\begin{equation}\label{eq: p-min}
\min\{A(x)>0\,|\,x\;\text{is a generalized}\;\Psi\hbox{-characteristic on}\;\partial D\}=(\min_{x\in\mathcal{A}_p}I_p)^{\frac{2}{p}}.
\end{equation}
We will prove this in four steps.

\noindent{\bf Step 1}.\quad{\it $\mu_p:=\inf_{x\in\mathcal{A}_p}I_p(x)$ is positive. }
It is easy to prove that
\begin{equation}\label{e:Brun.2}
\|x\|_{L^\infty}\le\widetilde{C}_1\|\dot{x}\|_{L^p}\quad\forall x\in\mathcal{F}_p
\end{equation}
for some constant $\widetilde{C}_1=\widetilde{C}_1(p)>0$. So for any $x\in\mathcal{A}_p$ we have
$$
2=2A_p(x)\le \|x\|_{L^q}\|\dot{x}\|_{L^p}\le \|x\|_{L^\infty}\|\dot{x}\|_{L^p}\le  \widetilde{C}_1\|\dot{x}\|_{L^p}^2,
$$
and thus $\|\dot{x}\|_{L^p}\ge \sqrt{2/\widetilde{C}_1 }$, where $1/p+1/q=1$.
Let $R_2$ be as in (\ref{e:dualestimate0}). These lead to
$$
I_p(x)\ge\left(\frac{1}{R_2}\right)^{p/2}\|\dot{x}\|_{L^p}^p\ge \widetilde{C}_2,
\quad\hbox{where}\quad\widetilde{C}_2=
\left(\frac{2}{R_2\widetilde{C}_1}\right)^{\frac{p}{2}}>0.
$$
\noindent{\bf Step 2}.\quad{\it There exists $u\in  \mathcal{A}_p$ such  that
$I_p(u)=\mu_p$, i.e. the infimum of $I_p$ on $\mathcal{A}_p$ can be attained by some $u\in  \mathcal{A}_p$. }
Let $(x_n)\subset\mathcal{A}_p$ be a sequence satisfying
 $\lim_{n\rightarrow+\infty}I_p(x_n)=\mu_p$.
 Then there exists a constant $\widetilde{C}_3>0$ such that
$$
\left(\frac{1}{R_2}\right)^{p/2}\|\dot{x}_n\|_{L^p}^p\le I_p(x_n)\le \widetilde{C}_3,\quad\forall n\in\mathbb{N}.
$$
By (\ref{e:Brun.2}) and the fact that $\|x\|_{L^p}\le\|x\|_{L^\infty}$, we deduce that $(x_n)$ is bounded in $W^{1,p}([0,1],\mathbb{R}^{2n})$. Note that $W^{1,p}([0,1])$
is reflexive for $p>1$. $(x_n)$ has  a subsequence, also denoted by $(x_n)$, which
converges weakly to some $u\in W^{1,p}([0,1],\mathbb{R}^{2n})$.
By Arzel\'{a}-Ascoli theorem, there also exists $\hat{u}\in C^{0}([0,1],\mathbb{R}^{2n})$ such that
$$
\lim_{n\rightarrow+\infty}\sup_{t\in [0,1]}|x_n(t)-\hat{u}(t)|=0.
$$
A standard argument yields
$u(t)=\hat{u}(t)$ almost everywhere. We may consider that $x_n$ converges uniformly to $u$.
 Hence $u(1)=\Psi u(0)$ and $u(0)\in E_1^{\bot}$. As in Step 2 of \cite[Section~4.1]{JinLu}, we also have $A_p(u)=1$, and so $u\in\mathcal{A}_p$. Standard argument in convex analysis shows that there exists $\omega\in L^q([0,1],\mathbb{R}^{2n})$ such that $\omega(t)\in\partial (H_D^{\ast})^{\frac{p}{2}}(-J\dot{u}(t))$
 almost everywhere. These lead to
$$
I_p(u)-I_p(x_n)\le \int_0^1\langle \omega(t),-J(\dot{u}(t)-\dot{x}_n(t))\rangle dt\rightarrow 0 \quad\hbox{as}\quad n\rightarrow\infty ,
$$
since $x_n$ converges weakly to $u$. Hence
$\mu_p\le I_p(u)\le\lim_{n\rightarrow\infty}I_p(x_n)= \mu_p$.

\noindent{\bf Step 3}.\quad{\it There exists a generalized $\Psi$-characteristic on $\partial D$,
${x}^\ast:[0, 1]\rightarrow \partial D$,  such that $A({x}^\ast)=(\mu_p)^{\frac{2}{p}}$.} Since $u$ is the minimizer  of  $I_p|_{\mathcal{A}_p}$,
  applying Lagrangian multiplier theorem (cf. \cite[Theorem~6.1.1]{Cl83}) we get some $\lambda_p\in\mathbb{R}$ such that
$0\in\partial (I_p+\lambda_p A)(u)=\partial I_p(u)+\lambda_p A'(u)$.
This means that there exists some $\rho\in L^q([0,1],\mathbb{R}^{2n})$ satisfying
\begin{equation}\label{e:Brun.3-}
\rho(t)\in\partial (H_D^{\ast})^{\frac{p}{2}}(-J\dot{u}(t))\quad\hbox{a.e.}
\end{equation}
and
$$
\int_0^1\langle\rho(t),-J\dot{\zeta}(t)\rangle+\lambda_p\int_0^1\langle u(t),-J\dot{\zeta}(t)\rangle=0\quad\forall \zeta\in\mathcal{F}_p.
$$
From the latter we derive that for some ${\bf a}_0\in {\rm Ker}(\Psi-I)$,
\begin{equation}\label{e:Brun.3}
\rho(t)+\lambda_p u(t)={\bf a}_0,\quad\hbox{a.e..}\quad
\end{equation}
 Computing as in the case of $p=2$ (cf. Step 3 of \cite[Section~4.1]{JinLu}), we get that
$$
\lambda_p=-\frac{p}{2}\mu_p.
$$
Since $p>1$, $q=p/(p-1)>1$. From (\ref{e:Brunn.0}) we may derive that
 $(H_D^{\ast})^{\frac{p}{2}}=(\frac{h_D}{2})^p$ has the Legendre transformation
given by
\begin{eqnarray*}
\left(\frac{h_D^p}{2^p}\right)^{\ast}(x)
=\left(\frac{h_D^p}{p}\right)^{\ast}(\frac{2}{p^{\frac{1}{p}}}x)
=\frac{1}{q}j_D^q(\frac{2}{p^{\frac{1}{p}}}x)
=\frac{2^q}{qp^{\frac{q}{p}}}j_D^q(x)
=\frac{2^q}{qp^{q-1}}j_D^q(x).
\end{eqnarray*}
Using this and (\ref{e:Brun.3-})-(\ref{e:Brun.3}), we get that
$$
-J\dot{u}(t)\in\frac{2^q}{qp^{q-1}}\partial j_D^q(-\lambda_p u(t)+{\bf a}_0),\quad\hbox{a.e.}.
$$
Let $v(t):=-\lambda_p u(t)+{\bf a}_0$. Then
$$
-J\dot{v}(t)\in-\lambda_p\frac{2^q}{qp^{q-1}}\partial j_D^q(v(t))
\quad\hbox{and}\quad v(1)=\Psi v(0).
$$
This implies that $j_D^q(v(t))$ is a constant by \cite[Theorem~2]{Ku96}, and
$$
\frac{-2^{q-1}\lambda_p}{p^{q-1}}j_D^q(v(t))=\int^1_0\frac{-2^{q-1}\lambda_p}{p^{q-1}}j_D^q(v(t))dt
=\frac{1}{2}\int^1_0\langle-J\dot{v}(t),
v(t)\rangle dt=\lambda_p^2=\left(\frac{p\mu_p}{2}\right)^2
$$
by the Euler formula \cite[Theorem~3.1]{YangWei08}.
Therefore $j_D^q(v(t))=\left(\frac{p}{2}\right)^q\mu_p$ and
$$
A(v)=\frac{1}{2}\int^1_0\langle-J\dot{v}(t), v(t)\rangle dt=\lambda_p^2=\left(\frac{p\mu_p}{2}\right)^2.
$$
Let $x^{\ast}(t)=\frac{v(t)}{j_D(v(t))}$. Then $x^{\ast}$ is a generalized $\Psi$-characteristic on $\partial D$ with action
$$
A(x^{\ast})=\frac{1}{j_D^2(v(t))}A(v)=\mu_p^{\frac{2}{p}}.
$$
\noindent{\bf Step 4}. \quad{\it For any generalized $\Psi$-characteristic
on $\partial D$ with positive action, $y:[0,T]\rightarrow \partial D$, there holds $A(y)\ge \mu_p^{\frac{2}{p}}$.}  Since  \cite[Theorem~2.3.9]{Cl83} implies
$\partial j_D^q(x)=q(j_D(x))^{q-1}\partial j_D(x)$,
by \cite[Lemma~4.2]{JinLu}, after reparameterization we may assume that
$y\in W^{1,\infty}([0,T],\mathbb{R}^{2n})$ and satisfies
$$
j_D(y(t))\equiv 1\quad\hbox{and}\quad-J\dot{y}(t)\in\partial j_D^q(y(t))\quad\hbox{a.e. on}\;[0, T].
$$
 It follows that
\begin{equation}\label{e:Brun.4}
A(y)=\frac{qT}{2}.
\end{equation}
Similar to the case $p=2$, define $y^{\ast}:[0,1]\rightarrow \mathbb{R}^{2n}$,   $t\mapsto y^{\ast}(t)=a y(tT)+ {\bf b}$,
 where $a>0$ and ${\bf b}\in E_1$ are chosen so that $y^{\ast}\in\mathcal{A}_p$. Then (\ref{e:Brun.4})
 leads to
 \begin{equation}\label{e:Brun.5}
 1=A(y^{\ast})=a^2A(y)=\frac{a^2qT}{2}.
 \end{equation}
 Moreover, it is clear that
 $$
 -J\dot{y}^{\ast}(t)\in\frac{2^q}{qp^{q-1}}\partial(j_D^q)\left((aT)^{\frac{1}{q-1}}
 \frac{q^{\frac{1}{q-1}}p}{2^p}y(tT)\right).
 $$
We use this, (\ref{e:Brunn.0}) and
 the Legendre reciprocity formula (cf. \cite[Proposition~II.1.15]{Ek90})
 to derive
 \begin{eqnarray*}
 &&\frac{2^q}{qp^{q-1}}j_D^q ((aT)^{\frac{1}{q-1}}
 \frac{q^{\frac{1}{q-1}}p}{2^p}y(tT))+ \left(\frac{h_D^p}{2^p}\right)^{\ast}(-J\dot{y}^{\ast}(t))\\
 &=&\langle-J\dot{y}^{\ast}(t),(aT)^{\frac{1}{q-1}}
 \frac{q^{\frac{1}{q-1}}p}{2^p}y(tT)\rangle
 \end{eqnarray*}
 and hence
 \begin{eqnarray*}
 (H_D^{\ast}(-J\dot{y}^{\ast}(t)))^{\frac{p}{2}}
 &=&\left(\frac{h_D^p}{2^p}\right)^{\ast}(-J\dot{y}^{\ast}(t))\\
 &=&(aT)^p\frac{q^pp}{2^p}-(aT)^p\frac{q^{p-1}p}{2^p}\\
 &=&(aT)^p\frac{q^{p-1}p(q-1)}{2^p}\\
 &=&(aT)^p\frac{q^{p}}{2^p}\ge\mu_p.
 \end{eqnarray*}
 By Step 1 we get $I_p(y^\ast)\ge \mu_p$ and so $(aT)^p\frac{q^{p}}{2^p}\ge\mu_p$.
 This, (\ref{e:Brun.4}) and (\ref{e:Brun.5}) lead to
 $A(y)\ge \mu_p^{\frac{2}{p}}$.

 Summarizing the four steps we get (\ref{eq: p-min}) and hence (\ref{eq: p-rep}) is proved.
\end{proof}

\begin{remark}\label{rem:Brun.1.5}
{\rm
\begin{description}
\item[(i)] Checking Step~3, it is easily seen that
for a minimizer $u$ of  $I_p|_{\mathcal{A}_p}$ there exists ${\bf a}_0\in {\rm Ker}(\Psi-I)$
such that
$$
x^\ast(t)=\left(c^\Psi_{\rm EHZ}(D)\right)^{1/2}u(t)+ \frac{2}{p}\left(c^\Psi_{\rm EHZ}(D)\right)^{(1-p)/2}{\bf a}_0
$$
gives a generalized $\Psi$-characteristic on $\partial D$
with action $A(x^{\ast})=c^\Psi_{\rm EHZ}(D)$, namely, $x^\ast$ is a
$c^\Psi_{\rm EHZ}$-carrier for $\partial D$.
\item[(ii)] For a generalized $\Psi$-characteristic on $\partial D$ with action $A(x^{\ast})=c^\Psi_{\rm EHZ}(D)$, computation in
Step~4 implies that
$$
u(t)=\frac{x^\ast(tT)}{\sqrt{c_{\rm EHZ}^\Psi(D)}}+b=\frac{x^\ast(tT)}{\sqrt{A(x^\ast)}}+b,\quad \hbox{for some}\quad b\in E_1
$$
is a minimizer of $I_p|_{\mathcal{A}_p}$.
\end{description}}
\end{remark}

\begin{proof}[\it Proof of Proposition~\ref{prop:Brun.2}].
 Firstly, suppose $p_1\ge p_2>1$. Then $\mathcal{A}_{p_1}\subset \mathcal{A}_{p_2}$
and the first two steps in the proof of Proposition~\ref{prop:Brun.1} implies that
  $I_{p_1}|_{\mathcal{A}_{p_1}}$ has a minimizer $u\in \mathcal{A}_{p_1}$.
  It follows that
   \begin{eqnarray*}
  c^{\Psi}_{\rm EHZ}(D)&=&\left(\int_0^1(H_D^{\ast}(-J\dot{u}(t)))^{\frac{p_1}{2}}dt\right)^
  {\frac{2}{p_1}}\\
  &\ge& \left(\int_0^1(H_D^{\ast}(-J\dot{u}(t)))^{\frac{p_2}{2}}dt\right)^
  {\frac{2}{p_2}}\\
  &\ge & \inf_{x\in\mathcal{A}_{p_1}}\left(\int_0^1(H_D^{\ast}(-J\dot{x}(t)))^{\frac{p_2}{2}}dt\right)^
  {\frac{2}{p_2}}\\
  &\ge & \inf_{x\in\mathcal{A}_{p_2}}\left(\int_0^1(H_D^{\ast}(-J\dot{x}(t)))^{\frac{p_2}{2}}dt\right)^
  {\frac{2}{p_2}}\\
  &=&c^{\Psi}_{\rm EHZ}(D),
  \end{eqnarray*}
  where two equalities come from Lemma~\ref{prop:Brun.1} and the first inequality
  is because of H\"older's inequality.
  Hence the functional $\int_0^1(H_D^{\ast}(-J\dot{x}(t)))^{\frac{p_2}{2}}dt$ attains its minimum at $u$ on $\mathcal{A}_{p_1}$ and
  \begin{equation}\label{e:Brun.6}
  c^{\Psi}_{\rm EHZ}(D)=\min_{x\in\mathcal{A}_{p_1}}
  \left(\int_0^1(H_D^{\ast}(-J\dot{x}(t)))^{\frac{p_2}{2}}dt\right)^{\frac{2}{p_2}}.
  \end{equation}

Next, if  $p_2\ge p_1>1$, then $\mathcal{A}_{p_2}\subset \mathcal{A}_{p_1}$
and we have $u\in \mathcal{A}_{p_2}$ minimizing
  $I_{p_2}|_{\mathcal{A}_{p_2}}$ such  that
   \begin{eqnarray*}
  c^{\Psi}_{\rm EHZ}(D)&=&\left(\int_0^1(H_D^{\ast}(-J\dot{u}(t)))^{\frac{p_2}{2}}dt\right)^
  {\frac{2}{p_2}}\\
   &\ge & \inf_{x\in\mathcal{A}_{p_1}}\left(\int_0^1(H_D^{\ast}(-J\dot{x}(t)))^{\frac{p_2}{2}}dt\right)^
  {\frac{2}{p_2}}\\
  &\ge & \inf_{x\in\mathcal{A}_{p_1}}\left(\int_0^1(H_D^{\ast}(-J\dot{x}(t)))^{\frac{p_1}{2}}dt\right)^
  {\frac{2}{p_1}}\\
  &=&c^{\Psi}_{\rm EHZ}(D).
  \end{eqnarray*}
This yields (\ref{e:Brun.6}) again.

  Finally, for $p_2=1$ and $p_1>1$ let $u\in \mathcal{A}_{p_1}$ minimize
  $I_{p_1}|_{\mathcal{A}_{p_1}}$. It is clear that
  \begin{eqnarray}\label{e:Brun.7}
  c^{\Psi}_{\rm EHZ}(D)&=&\left(\int_0^1(H_D^{\ast}(-J\dot{u}(t)))^{\frac{p_1}{2}}dt\right)^
  {\frac{2}{p_1}}\nonumber\\
  &\ge& \left(\int_0^1(H_D^{\ast}(-J\dot{u}(t)))^{\frac{1}{2}}dt\right)^{2}\nonumber\\
  &\ge & \inf_{x\in\mathcal{A}_{p_1}}\left(\int_0^1(H_D^{\ast}(-J\dot{x}(t)))^{\frac{1}{2}}dt
  \right)^{2}
  \end{eqnarray}
  Let $R_2$ be as in (\ref{e:dualestimate0}).
 Then
 $$
 (H_D^{\ast}(-J\dot{x}(t)))^{\frac{p}{2}}\le (R_2|\dot{x}(t)|^2)^{\frac{p}{2}}\le (R_2+1)^{\frac{p_1}{2}}
 |\dot{x}(t)|^{p_1}
 $$
 for any $1\le p\le p_1$. By (\ref{e:Brun.6})
 $$
  c^{\Psi}_{\rm EHZ}(D)=\min_{x\in\mathcal{A}_{p_1}}
  \left(\int_0^1(H_D^{\ast}(-J\dot{x}(t)))^{\frac{p}{2}}dt\right)^{\frac{2}{p}},\quad 1<p\le p_1.
 $$
    Letting $p\downarrow 1$ and using Lebesgue dominated convergence theorem we get
  $$
  c^{\Psi}_{\rm EHZ}(D)\le \inf_{x\in\mathcal{A}_{p_1}}\left(\int_0^1(H_D^{\ast}(-J\dot{x}(t)))^{\frac{1}{2}}dt\right)^
  {2}.
  $$
  This and (\ref{e:Brun.7}) show  that the functional $\mathcal{A}_{p_1}\ni x\mapsto\int_0^1(H_D^{\ast}(-J\dot{x}(t)))^{\frac{1}{2}}dt$ attains its minimum at $u$ and
 $$
  c^{\Psi}_{\rm EHZ}(D)=\min_{x\in\mathcal{A}_{p_1}}
  \left(\int_0^1(H_D^{\ast}(-J\dot{x}(t)))^{\frac{1}{2}}dt\right)^{2}.
  $$
  Proposition~\ref{prop:Brun.2} is proved.
\end{proof}

\begin{proof}[\it Proof of Theorem~\ref{th:Brun}]
Choose a real $p_1>1$. Then for $p\ge 1$ Proposition~\ref{prop:Brun.2} implies
\begin{eqnarray}\label{e:Brun.8}
  c^{\Psi}_{\rm EHZ}(D+_pK)^{\frac{p}{2}}
  &=&\min_{x\in\mathcal{A}_{p_1}}\frac{1}{2^p}\int_0^1
  (h_{D+_pK}(-J\dot{x}))^{p}dt\\
  &=&\min_{x\in\mathcal{A}_{p_1}}\frac{1}{2^p}\int_0^1
  ((h_{D}(-J\dot{x}))^{p}+(h_{K}(-J\dot{x}))^{p})dt\nonumber\\
  &\ge& \min_{x\in\mathcal{A}_{p_1}}\frac{1}{2^p}\int_0^1
  (h_{D}(-J\dot{x}))^{p}+\min_{x\in\mathcal{A}_{p_1}}\frac{1}{2^p}\int_0^1
  (h_{K}(-J\dot{x}))^{p}dt\nonumber\\
  &=&c^{\Psi}_{\rm EHZ}(D)^{\frac{p}{2}}+c^{\Psi}_{\rm EHZ}(K)^{\frac{p}{2}}.
\end{eqnarray}

Now suppose that $p\ge 1$ and there exist $c^\Psi_{\rm EHZ}$ carriers
$\gamma_D:[0, T]\to\partial D$ and $\gamma_K:[0, T]\to\partial K$
satisfying $\gamma_D=\alpha\gamma_K+{\bf b}$
for some $\alpha\in\mathbb{R}\setminus\{0\}$ and
some ${\bf b}\in{\rm Ker}(\Psi-I_{2n})$. We will prove the equality in (\ref{e:BrunA}) holds.
(\ref{e:action1}) implies $A(\gamma_D)=\alpha^2 A(\gamma_K)$.
Moreover by Remark~\ref{rem:Brun.1.5}(ii)
 for suitable vectors ${\bf b}_D, {\bf b}_K\in {\rm Ker}(\Psi-I_{2n})$
\begin{eqnarray*}
z_D(t)=\frac{1}{\sqrt{A(\gamma_D)}}\gamma_D(Tt)+{\bf b}_D\quad\hbox{and}\quad
z_K(t)=\frac{1}{\sqrt{A(\gamma_K)}}\gamma_K(Tt)+{\bf b}_K
\end{eqnarray*}
in $\mathcal{A}_{p_1}$ satisfy
\begin{eqnarray}\label{e:Brun.8.1}
&&c^{\Psi}_{\rm EHZ}(D)^{\frac{p}{2}}=\min_{x\in\mathcal{A}_{p_1}}\frac{1}{2^p}\int_0^1
  (h_{D}(-J\dot{x}))^{p}dt=\frac{1}{2^p}\int_0^1
  (h_{D}(-J\dot{z}_D))^{p}dt,\\
  &&c^{\Psi}_{\rm EHZ}(K)^{\frac{p}{2}}=  \min_{x\in\mathcal{A}_{p_1}}\frac{1}{2^p}\int_0^1
  (h_{K}(-J\dot{x}))^{p}dt=\frac{1}{2^p}\int_0^1
  (h_{K}(-J\dot{z}_K))^{p}dt.\label{e:Brun.8.2}
\end{eqnarray}

It follows that $\dot{z}_D(t)=\alpha\left(\frac{A(\gamma_K)}{A(\gamma_D)}\right)^{1/2}\dot{z}_K=\dot{z}_K$
because $A(\gamma_D)=\alpha^2 A(\gamma_K)$. Then (\ref{e:Brun.8.1}) and (\ref{e:Brun.8.2}) lead to
\begin{eqnarray*}
&&c^{\Psi}_{\rm EHZ}(D)^{\frac{p}{2}}+c^{\Psi}_{\rm EHZ}(K)^{\frac{p}{2}}\\
&=&\frac{1}{2^p}\int_0^1((h_{D}(-J\dot{z}_D))^{p}+(h_{K}(-J\dot{z}_D))^{p})dt\\
&=&\frac{1}{2^p}\int_0^1h_{D+_pK}(-J\dot{z}_D)^pdt\\
&\ge & \min_{x\in\mathcal{A}_{p_1}}\frac{1}{2^p}\int_0^1
  (h_{D+_pK}(-J\dot{x}))^{p}dt\\
&=&  c^{\Psi}_{\rm EHZ}(D+_pK)^{\frac{p}{2}}.
\end{eqnarray*}
Combined with (\ref{e:Brun.8}) we get
\begin{eqnarray*}
  c^{\Psi}_{\rm EHZ}(D+_pK)^{\frac{p}{2}}
  =c^{\Psi}_{\rm EHZ}(D)^{\frac{p}{2}}+c^{\Psi}_{\rm EHZ}(K)^{\frac{p}{2}}.
\end{eqnarray*}

Now suppose that $p>1$ and the equality in (\ref{e:BrunA}) holds.
We may require that the above $p_1$ satisfies $1<p_1<p$.
By Proposition~\ref{prop:Brun.2} there exists $u\in\mathcal{A}_{p_1}$ such that
\begin{eqnarray*}
  c^{\Psi}_{\rm EHZ}(D+_pK)^{\frac{p}{2}}=\frac{1}{2^p}\int_0^1
  \left((h_{D+_pK}(-J\dot{u}))\right)^{p}dt.
  \end{eqnarray*}
The equality in (\ref{e:BrunA}) yields
 \begin{eqnarray*}
  &&\frac{1}{2^p}\int_0^1
  ((h_{D}(-J\dot{u}))^{p}+(h_{K}(-J\dot{u}))^{p})dt\\
  &=& \min_{x\in\mathcal{A}_{p_1}}\frac{1}{2^p}\int_0^1
  (h_{D}(-J\dot{x}))^{p}dt+\min_{x\in\mathcal{A}_{p_1}}\frac{1}{2^p}\int_0^1
  (h_{K}(-J\dot{x}))^{p}dt
\end{eqnarray*}
 and thus
\begin{eqnarray*}
&&c^\Psi_{\rm EHZ}(D)^{\frac{p}{2}}=\min_{x\in\mathcal{A}_{p_1}}\frac{1}{2^p}\int_0^1
  (h_{D}(-J\dot{x}))^{p}dt=\frac{1}{2^p}\int_0^1
  (h_{D}(-J\dot{u}))^{p}dt\quad\hbox{and}\\
 &&c^\Psi_{\rm EHZ}(K)^{\frac{p}{2}}=\min_{x\in\mathcal{A}_{p_1}}\frac{1}{2^p}\int_0^1
  (h_{K}(-J\dot{x}))^{p}dt=\frac{1}{2^p}\int_0^1
  (h_{K}(-J\dot{u}))^{p}dt.
\end{eqnarray*}
These and  Propositions~\ref{prop:Brun.1}, \ref{prop:Brun.2}  and H\"older's inequality lead to
\begin{eqnarray*}
\min_{x\in\mathcal{A}_{p_1}}\left(\int_0^1
  (h_{D}(-J\dot{x}))^{p_1}dt\right)^{\frac{1}{p_1}}&=& 2(c^{\Psi}_{\rm EHZ}(D))^{\frac{1}{2}} \\   &=&\min_{x\in\mathcal{A}_{p_1}}\left(\int_0^1
  (h_{D}(-J\dot{x}))^{p}dt\right)^{\frac{1}{p}}\\
  &=&\left(\int_0^1(h_{D}(-J\dot{u}))^{p}dt\right)^{\frac{1}{p}}\ge
  \left(\int_0^1
  (h_{D}(-J\dot{u}))^{p_1}dt\right)^{\frac{1}{p_1}},\\
 \min_{x\in\mathcal{A}_{p_1}}\left(\int_0^1
  (h_{K}(-J\dot{x}))^{p_1}dt\right)^{\frac{1}{p_1}}&=&2(c^{\Psi}_{\rm EHZ}(K))^{\frac{1}{2}} \\ &=&\min_{x\in\mathcal{A}_{p_1}}\left(\int_0^1
  (h_{K}(-J\dot{x}))^{p}dt\right)^{\frac{1}{p}}\\
  &=&\left(\int_0^1(h_{K}(-J\dot{u}))^{p}dt\right)^{\frac{1}{p}}\ge
  \left(\int_0^1
  (h_{K}(-J\dot{u}))^{p_1}dt\right)^{\frac{1}{p_1}}.
 \end{eqnarray*}
  It follows that
 \begin{eqnarray*}
 2(c^\Psi_{\rm EHZ}(D))^{\frac{1}{2}}&=&
  \left(\int_0^1(h_{D}(-J\dot{u}))^{p}dt\right)^{\frac{1}{p}}=
  \left(\int_0^1(h_{D}(-J\dot{u}))^{p_1}dt\right)^{\frac{1}{p_1}},\\
 2(c^\Psi_{\rm EHZ}(K))^{\frac{1}{2}}&=&
  \left(\int_0^1(h_{K}(-J\dot{u}))^{p}dt\right)^{\frac{1}{p}}
  =\left(\int_0^1(h_{K}(-J\dot{u}))^{p_1}dt\right)^{\frac{1}{p_1}}.
 \end{eqnarray*}
%
By Remark~\ref{rem:Brun.1.5}(i) there are
 ${\bf a}_D, {\bf a}_K\in {\rm Ker}(\Psi-I_{2n})$ such that
\begin{eqnarray*}
&&\gamma_D(t)=\left(c^\Psi_{\rm EHZ}(D)\right)^{1/2}u(t)+ \frac{2}{p_1}\left(c^\Psi_{\rm EHZ}(D)\right)^{(1-p_1)/2}{\bf a}_D,\\
&&\gamma_K(t)=\left(c^\Psi_{\rm EHZ}(K)\right)^{1/2}u(t)+ \frac{2}{p_1}\left(c^\Psi_{\rm EHZ}(K)\right)^{(1-p_1)/2}{\bf a}_K
\end{eqnarray*}
are  $c^\Psi_{\rm EHZ}$ carriers for $\partial D$ and $\partial K$, respectively.
Clearly, they coincide up to  dilation and translation in
${\rm Ker}(\Psi-I_{2n})$.  Theorem~\ref{th:Brun} is proved.
\end{proof}

\subsection{Some interesting consequences of Theorem~\ref{th:Brun}}\label{sec:Brunn.2}


Since $D+_1K=D+K=\{x+y\,|\, x\in D\;{\rm and}\;y\in K\}$ we have:  

\begin{corollary}\label{cor:Brun.1}
 Let $\Psi\in{\rm Sp}(2n,\mathbb{R})$, and let
 $D, K\subset \mathbb{R}^{2n}$ be two convex bodies containing fixed points of $\Psi$
  in their interiors.   Then
   \begin{description}
\item[(i)]
\begin{equation}\label{e:BrunA+}
   \left(c^{\Psi}_{\rm EHZ}(D+K)\right)^{\frac{1}{2}}\ge \left(c^{\Psi}_{\rm EHZ}(D)\right)^{\frac{1}{2}}+ \left(c^{\Psi}_{\rm EHZ}(K)\right)^{\frac{1}{2}},
   \end{equation}
   and the equality holds if  there exist
    $c^\Psi_{\rm EHZ}$-carriers for $D$ and $K$ which coincide up to  dilation and translation by elements
    in ${\rm Ker}(\Psi-I_{2n})$.
\item[(ii)] For  $x,y\in {\rm Fix}(\Psi)$, if both
${\rm Int}(D)\cap{\rm Fix}(\Psi)-x$ and ${\rm Int}(D)\cap{\rm Fix}(\Psi)-y$
are intersecting with ${\rm Int}(K)$, then
 \begin{eqnarray}\label{e:BrunB}
&&\lambda\left(c^\Psi_{\rm EHZ}(D\cap (x+K))\right)^{1/2}+
(1-\lambda)\left(c^\Psi_{\rm EHZ}(D\cap (y+K))\right)^{1/2}\nonumber\\
&\le&\left(c^\Psi_{\rm EHZ}(D\cap(\lambda x+(1-\lambda)y+K))\right)^{1/2},\quad\forall\, 0\le \lambda\le 1.
\end{eqnarray}
  In particular, if $D$ and $K$ are centrally symmetric, i.e., $-D=D$ and $-K=K$, then
  \begin{equation}\label{e:BrunB+}
  c^\Psi_{\rm EHZ}(D\cap(x+K))\le c^\Psi_{\rm EHZ}(D\cap K),\quad\forall x\in{\rm Fix}(\Psi).
  \end{equation}
\end{description}
\end{corollary}

\begin{proof}
{\bf (i)} Indeed, let $p\in{\rm Fix}(\Psi)\cap{\rm Int}(D)$ and $q\in{\rm Fix}(\Psi)\cap{\rm Int}(K)$. Then  (\ref{e:BrunA}) implies
\begin{eqnarray*}
 \left(c^{\Psi}_{\rm EHZ}(D+K-p-q)\right)^{\frac{1}{2}}&=&  \left(c^{\Psi}_{\rm EHZ}((D-p)+(K-q))\right)^{\frac{1}{2}}\\
 &\ge& \left(c^{\Psi}_{\rm EHZ}(D-p)\right)^{\frac{1}{2}}+ \left(c^{\Psi}_{\rm EHZ}(K-q)\right)^{\frac{1}{2}}.
 \end{eqnarray*}
 For $z\in \mathbb{R}^{2n}$, consider the symplectomorphism
 $\phi_z:(\mathbb{R}^{2n},\omega_0)\to (\mathbb{R}^{2n},\omega_0),\;x\mapsto x-z$.
 Since $p$, $q$ and $p+q$ are all fixed points of $\Psi$, and
  $\phi_p$, $\phi_q$ and $\phi_{p+q}$
 commute with $\Psi$,  by Proposition~\ref{MonComf} it is clear that
\begin{eqnarray*}
&&c^{\Psi}_{\rm EHZ}(D+K-p-q)=c^{\Psi}_{\rm EHZ}(\phi_{p+q}(D+K))=c^{\Psi}_{\rm EHZ}(D+K),\\
&&c^{\Psi}_{\rm EHZ}(D-p)=c^{\Psi}_{\rm EHZ}(\phi_{p}(D))=c^{\Psi}_{\rm EHZ}(D),\\
&&c^{\Psi}_{\rm EHZ}(K-q)=c^{\Psi}_{\rm EHZ}(\phi_{q}(K))=c^{\Psi}_{\rm EHZ}(K).
\end{eqnarray*}
Other claims easily follow from the arguments therein.

\noindent{\bf (ii)} Since  $x,y\in {\rm Fix}(\Psi)$,  both
${\rm Int}(D)\cap{\rm Fix}(\Psi)-x$ and ${\rm Int}(D)\cap{\rm Fix}(\Psi)-y$
are intersecting with ${\rm Int}(K)$, we deduce that for any $0\le \lambda\le 1$
interiors of $\lambda(D\cap (x+K))$ and $(1-\lambda)(D\cap (y+K))$ contain fixed points of $\Psi$.
(\ref{e:BrunB}) follows from Proposition~\ref{MonComf} and (i) directly.

Suppose further that $D$ and $K$ are centrally symmetric, i.e., $-D=D$ and $-K=K$.
Then $D\cap(-x+K)=-(D\cap(x+K))$ and
$c^\Psi_{\rm EHZ}(-(D\cap (x+K)))=c^\Psi_{\rm EHZ}(D\cap (x+K))$
since the symplectomorphism $\mathbb{R}^{2n}\to \mathbb{R}^{2n},\,z\mapsto -z$
commutes with $\Psi$. Thus taking $y=-x$ and $\lambda=1/2$ in (\ref{e:BrunB})
leads to $c^\Psi_{\rm EHZ}(D\cap(x+K))\le c^\Psi_{\rm EHZ}(D\cap K)$.
\end{proof}

 Let $D$, $K$ and $\Psi$ be as in Corollary~\ref{cor:Brun.1}.
As in \cite{AAO08, AAO14} we may derive from Corollary~\ref{cor:Brun.1} that
the limit
\begin{equation}\label{e:BrunC}
\lim_{\varepsilon\to 0+}\frac{c^\Psi_{\rm EHZ}(D+\varepsilon K)-c^\Psi_{\rm EHZ}(D)}{\varepsilon}
\end{equation}
exists, denoted by $d^\Psi_K(D)$. In fact, by the assumptions we can choose $p\in{\rm Fix}(\Psi)\cap{\rm Int}(D)$ and $q\in{\rm Fix}(\Psi)\cap{\rm Int}(K)$. Then $(K-q)\subset R(D-p)$ for some $R>0$ (since $0\in{\rm int}(D-q)$).
 Note that $p+\varepsilon q\in {\rm Fix}(\Psi)\cap{\rm Int}(D+\varepsilon K)$.
By the proof of Corollary~\ref{cor:Brun.1}(i) and Proposition~\ref{MonComf}(ii) we get
\begin{eqnarray*}
c^\Psi_{\rm EHZ}(D+\varepsilon K)-c^\Psi_{\rm EHZ}(D)&=&c^\Psi_{\rm EHZ}((D-p)+\varepsilon (K-q))-c^\Psi_{\rm EHZ}(D-p)\\
&\le&c^\Psi_{\rm EHZ}((D-p)+\varepsilon R(D-p))-c^\Psi_{\rm EHZ}(D-p)\\
&\le&(1+\varepsilon R)c^\Psi_{\rm EHZ}(D-p)-c^\Psi_{\rm EHZ}(D-p)\\
&=&\varepsilon Rc^\Psi_{\rm EHZ}(D)
\end{eqnarray*}
and therefore  that the function of $\varepsilon>0$  in (\ref{e:BrunC})
is bounded. This function is also decreasing by Corollary~\ref{cor:Brun.1}(i)
(see reasoning \cite[pages 21-22]{AAO08}). Hence the limit in (\ref{e:BrunC}) exists.

The number $d^\Psi_K(D)$ may be viewed as the rate of change of the function $D\mapsto c^\Psi_{\rm EHZ}(D)$ in the ``direction" $K$.
From Corollary~\ref{cor:Brun.1} we can estimate it as follows.

\begin{corollary}\label{cor:Brun.2}
 Let $D$, $K$ and $\Psi$ be as in Corollary~\ref{cor:Brun.1}. Then it holds that
\begin{eqnarray}\label{e:BrunD}
2(c^\Psi_{\rm EHZ}(D))^{1/2}(c^\Psi_{\rm EHZ}(K))^{1/2}\le d^\Psi_K(D)\le
\inf_{z_D}\int_0^1h_K(-J\dot{z}_D(t))dt,
\end{eqnarray}
where $z_D:[0,1]\to\partial D$ takes over all  $c^\Psi_{\rm EHZ}$-carriers  for $D$.
\end{corollary}

In \cite{AAO08, AAO14} ${\rm length}_{JK^\circ}(z_D)=\int_0^1j_{JK^\circ}(\dot{z}_D(t))dt$
is called the \textsf{length of $z_D$ with respect to the convex body $JK^\circ$.}
In the case $0\in{\rm int}(K)$, since $h_K(-Jv)=j_{JK^\circ}(v)$,  (\ref{e:BrunD}) implies
$$
d^\Psi_K(D)\le
\inf_{z_D}\int_0^1j_{JK^\circ}(\dot{z}_D(t))dt\quad
\hbox{and hence}\quad
c^\Psi_{\rm EHZ}(D)c^\Psi_{\rm EHZ}(K)\le\frac{1}{4}\inf_{z_D}({\rm length}_{JK^\circ}(z_D))^2.
$$
It is not hard to see that (\ref{e:BrunB+}) may not hold if one of $D$ and $K$ is not convex.
Therefore the symplectic capacities only show good behavior in the convex category.

\begin{proof}[\it Proof of Corollary~\ref{cor:Brun.2}]
The first inequality in (\ref{e:BrunD}) easily follows from Corollary~\ref{cor:Brun.1}(i).
In order to prove the second one let us
fix a real $p_1>1$. By Proposition~\ref{prop:Brun.2} we have $u\in \mathcal{A}_{p_1}$ such
that
  \begin{eqnarray}\label{e:Brun.9}
  (c^{\Psi}_{\rm EHZ}(D))^{\frac{1}{2}}=(c^{\Psi}_{\rm EHZ}(D-p))^{\frac{1}{2}}
 &=&\min_{x\in\mathcal{A}_{p_1}}\frac{1}{2}\int_0^1
  h_{D-p}(-J\dot{x}))\nonumber\\
  &=&\frac{1}{2}\int_0^1h_{D-p}(-J\dot{u}))
  \end{eqnarray}
and that for some ${\bf a}_0\in {\rm Ker}(\Psi-I_{2n})$
\begin{eqnarray}\label{e:Brun.10}
x^\ast(t)=\left(c^\Psi_{\rm EHZ}(D)\right)^{1/2}u(t)+ \frac{2}{p_1}\left(c^\Psi_{\rm EHZ}(D)\right)^{(1-p_1)/2}{\bf a}_0
 \end{eqnarray}
is  a $c^\Psi_{\rm EHZ}$ carrier for $\partial (D-p)$ by Remark~\ref{rem:Brun.1.5}.
Proposition~\ref{prop:Brun.2} also leads to
\begin{eqnarray}\label{e:Brun.11}
  (c^{\Psi}_{\rm EHZ}(D+ \varepsilon K))^{\frac{1}{2}}&=&(c^{\Psi}_{\rm EHZ}((D-p)+ \varepsilon (K-q)))^{\frac{1}{2}}\\
  &=&\min_{x\in\mathcal{A}_{p_1}}\frac{1}{2}\int_0^1
  (h_{D-p}(-J\dot{x})+ \varepsilon h_{K-q}(-J\dot{x}))\nonumber\\
  &\le& \frac{1}{2}\int_0^1
  h_{D-p}(-J\dot{u})+\frac{\varepsilon}{2}\int_0^1
  h_{K-q}(-J\dot{u})\nonumber\\
  &=&(c^{\Psi}_{\rm EHZ}(D,\omega_0))^{\frac{1}{2}}+\frac{\varepsilon}{2}\int_0^1
  h_{K-q}(-J\dot{u})
\end{eqnarray}
because of (\ref{e:Brun.9}). Let $z_D(t)=x^\ast(t)+p$ for $0\le t\le 1$. Since $q$ and ${\bf a}_0$
are fixed points of $\Psi$ it is easily checked that $z_D$ is
a $c^\Psi_{\rm EHZ}$ carrier for $\partial D$. From (\ref{e:Brun.11}) it follows that
\begin{equation}\label{e:Brun.12}
\frac{(c^\Psi_{\rm EHZ}(D+\varepsilon K))^{\frac{1}{2}}-(c^\Psi_{\rm EHZ}(D))^{\frac{1}{2}}}{\varepsilon}\le
\frac{1}{2}\left(c^\Psi_{\rm EHZ}(D)\right)^{-\frac{1}{2}}
\int_0^1h_{K-q}(-J\dot{z}_D).
\end{equation}
Since $h_{K-q}(-J\dot{z}_D)=h_{K}(-J\dot{z}_D)+\langle q, J\dot{z}_D\rangle$ (see page 37 and Theorem~1.7.5 in \cite{Sch93}) and
$$
\int_0^1\langle q, J\dot{z}_D\rangle=\langle q, J(z_D(1)-z_D(0))\rangle=-\langle Jq, \Psi z_D(0)\rangle+\langle Jq, z_D(0)\rangle=0
$$
(by the fact $\Psi^tJ=J\Psi^{-1}$), letting $\varepsilon\to 0+$ in (\ref{e:Brun.12})
we arrive at the second inequality in (\ref{e:BrunD}).
\end{proof}


\section{Classification of $(A, \Delta, \Lambda)$-billiard trajectories  and related properties of proper trajectories}\label{sec:BillT-1}
\setcounter{equation}{0}


In this section, we give the classification of $(A, \Delta, \Lambda)$-billiard trajectories, related properties of proper trajectories,
 the relation between $A$-billiard trajectories in $\Delta$ and $(A,\Delta,B^n)$-billiard trajectories.
Moreover, on the base of the latter we prove that $\xi^A(\Delta)$ provides a lower bound of lengths of $A$-billiard trajectory in $\Delta$.

\begin{proposition}\label{prop:billT.0}
Let $A$, $\Delta$ and $\Lambda$ be as in (\ref{e:linesymp2}).
  \begin{description}
    \item[(i)]If both $\Delta$ and $\Lambda$ are also strictly convex (i.e.,
they have strictly positive Gauss curvatures at every point of their boundaries),
then every $(A, \Delta, \Lambda)$-billiard trajectory is either proper or  gliding.
\item[(ii)] Every proper $(A, \Delta, \Lambda)$-billiard trajectory
$\gamma:[0, T]\to \partial(\Delta\times\Lambda)$
cannot be contained in $\Delta\times\partial \Lambda$ or $\partial\Delta\times\Lambda$.
Consequently, $\gamma^{-1}(\partial\Delta\times\partial\Lambda)$ contains at least
a point in $(0, T)$.
  \end{description}
\end{proposition}

\begin{remark}
  {\rm If the condition "proper" in (ii) in the above claim is dropped, then "$\Delta\times\partial \Lambda$ or $\partial\Delta\times\Lambda$" should
  changed into "${\rm Int}(\Delta)\times\partial \Lambda$ or $\partial\Delta\times{\rm Int}(\Lambda)$".}
\end{remark}

\begin{proof}[{\it Proof of Proposition~\ref{prop:billT.0}}]
(i) can be obtained form Proposition~2.12 in \cite{AAO14}. Let us prove (ii).
By the definition we may assume that $\Delta\subset\mathbb{R}^n_q$ and $\Lambda\subset\mathbb{R}^n_p$
  contain the origin in their interiors.
We only need to prove that every proper $(A, \Delta, B^n)$-billiard trajectory
cannot be contained in $\Delta\times\partial \Lambda$. (Another case may be proved
with the same arguments.)
Otherwise, let $\gamma=(\gamma_q,\gamma_p):[0, T]\to\partial(\Delta\times \Lambda)$
be such a trajectory, that is, $\gamma([0,T])\subset\Delta\times\partial \Lambda$.
Then $\gamma^{-1}(\partial\Delta\times\partial \Lambda)$ is finite (including empty)
and there holds
$$
\dot{\gamma}(t)=(\dot{\gamma}_q(t), \dot{\gamma}_p(t))=(\kappa\nabla j_{\Lambda}({\gamma}_p(t)),0)          \quad\forall t\in [0, T]\setminus\gamma^{-1}(\partial\Delta\times\partial \Lambda)
$$
for some positive constant $\kappa$.
It follows that $\gamma_p$ is constant on each component of
$[0, T]\setminus\gamma^{-1}(\partial\Delta\times\partial \Lambda)$, and so constant
on $[0, T]\setminus\gamma^{-1}(\partial\Delta\times\partial \Lambda)$ by continuity of $\gamma$.
 Hence $\gamma_p\equiv p_0\in\partial \Lambda$,
 and so $\gamma_q(t)=q_0+\kappa t\nabla j_{\Lambda}(p_0)$ on $[0, T]$, where $q_0=\gamma_q(0)$. Now
 $$
 (q_0+\kappa T\nabla j_{\Lambda}(p_0), p_0)=\gamma(T)=\Psi_A\gamma(0)=(A\gamma_q(0), (A^t)^{-1}\gamma_p(0))=
 (Aq_0, (A^t)^{-1}p_0).
 $$
This implies that $A^tp_0=p_0$ and $q_0-Aq_0=-\kappa T\nabla j_{\Lambda}(p_0)$.
The former equality leads to $\langle p_0, v-Av\rangle=0\;\forall v\in\mathbb{R}^n$.
Combing this with the latter equality we obtain
 $\langle p_0, \nabla j_{\Lambda}(p_0)\rangle=0$.
This implies $j_{\Lambda}(p_0)=0$ and so $p_0=0$,
which contradicts $p_0\in\partial \Lambda$ since $0\in{\rm int}(\Lambda)$.
\end{proof}


Recall that the action of an $(A, \Delta,\Lambda)$-billiard trajectory $\gamma$ is given by (\ref{e:action1}).
The length of an $A$-billiard trajectory $\sigma:[0,T]\to\Delta$ is given by
$$
L(\sigma):=\sum_{i=0}^n\|q_{j+1}-q_j\|,
$$
with
$$q_0=\sigma(0),\; q_1=\sigma(t_i),\; \cdots,\; q_{m-1}=\sigma(t_{m-1}), \;q_m=\sigma(T),$$
where
$$
\{t_1,\cdots,t_{m-1}\}:=\mathcal{B}_{\sigma}
$$
is the finite set in Definition~\ref{def:billT.1}. Here $\|\cdot\|$  is the Euclid norm in $\mathbb{R}^n$.

The following proposition gives the relation between $A$-billiard trajectories in $\Delta$ and $(A,\Delta,B^n)$-billiard trajectories.

\begin{proposition}\label{prop:billT.5}
For a smooth convex body in $\Delta\subset\mathbb{R}^n$ and $A\in{\rm O}(n)$
satisfying ${\rm Fix}(A)\cap{\rm Int}(\Delta)\ne\emptyset$,
every $A$-billiard trajectory in  $\Delta$,  $\sigma:[0,T]\to \Delta$,
is the projection to  $\Delta$ of a proper $(A, \Delta, B^n)$-billiard trajectory
whose action is equal to the length of $\sigma$.
\end{proposition}

\begin{proof}
By the definitions we only need to consider the case that $0\in{\rm Int}(\Delta)$.
Let $\sigma:[0,T]\to \Delta$ be a $A$-billiard trajectory in  $\Delta$ with
$\mathscr{B}_\sigma=\{t_1<\cdots<t_k\}\subset (0, T)$ as in Definition~\ref{def:billT.2}.
Then $|\dot\sigma(t)|$ is equal to a positive constant $\kappa$ in $(0,T)\setminus\mathscr{B}_\sigma$.

\underline{Suppose that (ABiii) occurs}.  Define
\begin{eqnarray*}
&&\alpha_1(t)=(\sigma(t), -\frac{1}{\kappa}\dot\sigma^+(0)),\quad 0\le t\le t_1,\\
&&\beta_1(t)=(\sigma(t_1), -\frac{1}{\kappa}\dot\sigma^+(0)+ \frac{t}{\kappa}(\dot\sigma^-(t_1)-\dot\sigma^+(t_1)),\quad 0\le t\le 1.
\end{eqnarray*}
Since the second equality in (\ref{e:bill1}) implies that $\dot\sigma^-(t_i)-\dot\sigma^+(t_i)$
is an outer normal vector to $\partial\Delta$ at $\sigma(t_i)$ for each $t_i\in\mathscr{B}_\sigma$,
it is easily checked that both are generalized characteristics on $\partial(\Delta\times\Lambda)$ and
$\alpha_1(t_1)=\beta_1(0)$.
Similarly, define
\begin{eqnarray*}
&&\alpha_2(t)=(\sigma(t), -\frac{1}{\kappa}\dot\sigma^+(t_1)),\quad t_1\le t\le t_2,\\
&&\beta_2(t)=(\sigma(t_1), -\frac{1}{\kappa}\dot\sigma^+(t_1)+ \frac{t}{\kappa}(\dot\sigma^-(t_2)-\dot\sigma^+(t_2)),\quad 0\le t\le 1,\\
&&\vdots\\
&&\alpha_k(t)=(\sigma(t), -\frac{1}{\kappa}\dot\sigma^+(t_{k-1})),\quad t_{k-1}\le t\le t_k,\\
&&\beta_k(t)=(\sigma(t_{k-1}), -\frac{1}{\kappa}\dot\sigma^+(t_{k-1})+ \frac{t}{\kappa}(\dot\sigma^-(t_k)-\dot\sigma^+(t_k)),\quad 0\le t\le 1,\\
&&\alpha_{k+1}(t)=(\sigma(t), -\frac{1}{\kappa}\dot\sigma^+(t_{k}))=
(\sigma(t), -\frac{1}{\kappa}\dot\sigma^-(T)),\quad t_{k}\le t\le T.
\end{eqnarray*}
Then $\beta_1(1)=\alpha_2(t_1)$, $\alpha_2(t_2)=\beta_2(0)$, $\cdots$,
$\beta_k(1)=\alpha_{k+1}(t_k)$, that is, $\alpha_1\beta_1\cdots\alpha_k\beta_k\alpha_{k+1}$
is a path. Note also that
$$
\alpha_{k+1}(T)=(\sigma(T), -\frac{1}{\kappa}\dot\sigma^-(T))=(A\sigma(0), -\frac{1}{\kappa}A\dot\sigma^+(0))=
\Psi_A\alpha_1(0)
$$
by (\ref{e:bill5}). Hence $\gamma:=\alpha_1\beta_1\cdots\alpha_k\beta_k\alpha_{k+1}$
is a generalized $\Psi_A$-characteristic on $\partial(\Delta\times\Lambda)$.
Clearly, $\beta_1,\cdots,\beta_k$ all have zero actions. So 
$$
A(\gamma)=\sum_{i=0}^{k+1}\int_{t_i}^{t_{i+1}}\langle -\dot{\sigma}(t),-\frac{1}{\kappa}\dot{\sigma}^+(t_{i})\rangle_{\mathbb{R}^n}dt
=\kappa T=L(\sigma).
$$

\underline{Suppose that (ABiv) occurs}.   Let $\alpha_i$ and $\beta_j$ be defined
as above for  $i=1,\cdots,k+1$ and $j=1,\cdots,k$.
If (\ref{e:bill5}) holds, we also define $\gamma$ as above, and
get a generalized $\Psi_A$-characteristic on $\partial(\Delta\times\Lambda)$.

If (\ref{e:bill6}) occurs, we also need to define
\begin{eqnarray*}
\beta_0(t)=(\sigma(0), -\frac{1}{\kappa}\dot\sigma^-(0)+ \frac{t}{\kappa}(\dot\sigma^-(0)-\dot\sigma^+(0)),\quad 0\le t\le 1.
\end{eqnarray*}
By (\ref{e:bill4}), $\dot\sigma^-(0)-\dot\sigma^+(0)$ is an outer normal vector to $\partial\Delta$
at $\sigma(0)$. It is easy to see that $\beta_0$ is a
generalized characteristic on $\partial(\Delta\times\Lambda)$ satisfying $\beta_0(1)=\alpha_1(0)$. Moreover
$$
\Psi_A\beta_0(0)=\Psi_A(\sigma(0), -\frac{1}{\kappa}\dot\sigma^-(0))=
(A\sigma(0), -\frac{1}{\kappa}A\dot\sigma^-(0))=(\sigma(T), -\frac{1}{\kappa}\dot\sigma^-(T))=\alpha_{k+1}(T)
$$
by (\ref{e:bill6}). Thus $\gamma:=\beta_0\alpha_1\beta_1\cdots\alpha_k\beta_k\alpha_{k+1}$
is a generalized $\Psi_A$-characteristic on $\partial(\Delta\times\Lambda)$.

\underline{Suppose that (ABv) occurs}. If
(\ref{e:bill5}) holds, we define $\gamma$ as in the case of (ABv).
When (\ref{e:bill7}) occurs, we  need to define
\begin{eqnarray*}
\beta_{k+1}(t)=(\sigma(T), -\frac{1}{\kappa}\dot\sigma^-(T)+ \frac{t}{\kappa}(\dot\sigma^-(T)-\dot\sigma^+(T)),\quad 0\le t\le 1.
\end{eqnarray*}
Then $\gamma:=\alpha_1\beta_1\cdots\alpha_k\beta_k\alpha_{k+1}\beta_{k+1}$
is a generalized $\Psi_A$-characteristic on $\partial(\Delta\times\Lambda)$.

\underline{Suppose that (ABvi) occurs}. If
(\ref{e:bill5}) or (\ref{e:bill6}) or (\ref{e:bill7}) holds, we define
$$
\gamma:=\alpha_1\beta_1\cdots\alpha_k\beta_k\alpha_{k+1},\quad\hbox{or}\quad
\gamma:=\beta_0\alpha_1\beta_1\cdots\alpha_k\beta_k\alpha_{k+1},\quad\hbox{or}\quad
\gamma:=\alpha_1\beta_1\cdots\alpha_k\beta_k\alpha_{k+1}\beta_{k+1}.
$$
Finally, if (\ref{e:bill8}) holds, we define
$\gamma:=\beta_0\alpha_1\beta_1\cdots\alpha_k\beta_k\alpha_{k+1}\beta_{k+1}$.
\end{proof}

However, under the assumptions of Proposition~\ref{prop:billT.5} we cannot affirm
that the projection to  $\Delta$ of a proper $(A, \Delta, B^n)$-billiard trajectory
is an $A$-billiard trajectory in  $\Delta$.

\begin{proposition}\label{prop:billT.6}
Let $\Delta\subset\mathbb{R}^n$ be a smooth convex body  and $A\in{\rm O}(n)$
satisfy ${\rm Fix}(A)\cap{\rm Int}(\Delta)\ne\emptyset$. Then
 it holds that
$$
\xi^A(\Delta)\le\inf\{L(\sigma)\,|\,\hbox{$\sigma$ is an  $A$-billiard trajectory  in $\Delta$}\}.
$$
\end{proposition}
\begin{proof}
This may directly follow from Proposition~\ref{prop:billT.5}, Remark\ref{rem:ADL-BT}(i) and Theorem~\ref{th:convex}.
\end{proof}

The statement about relation between the action of a proper $(A, \Delta,B^n)$-billiard trajectory and the
length of its projection to $\Delta$ in Proposition~\ref{prop:billT.5} is a special case of the following proposition.
When $A=I_n$ it was showed in \cite[(7)]{AAO14}.

\begin{proposition}\label{prop:billT.2}
Let $A$, $\Delta$ and $\Lambda$ satisfy (\ref{e:linesymp2}).
If $\gamma:[0, T]\to \partial(\Delta\times\Lambda)$
is a proper $(A, \Delta, \Lambda)$-billiard trajectory
with $\gamma^{-1}(\partial\Delta\times\partial\Lambda)\cap (0,T)=\{t_1<\cdots<t_m\}$, then
the action of $\gamma$ is given by
\begin{equation}\label{e:linesymp4}
A(\gamma)=\sum^{m}_{j=0}h_\Lambda(q_j-q_{j+1})
\end{equation}
with $q_j=\pi_q(\gamma(t_j))$, $j=0,\cdots,m+1$, where $t_0=0$, $t_{m+1}=T$ and $q_{m+1}=Aq_0$.
In particular, if $\Lambda=B^n(\tau)$ for $\tau>0$ and $L(\pi_q(\gamma))$ denotes
   the length of the projection of $\gamma$ in $\Delta$ then
\begin{equation}\label{e:linesymp4.1}
A(\gamma)=\tau\sum^{m}_{j=0}\|q_{j+1}-q_j\|=\tau L(\pi_q(\gamma))
\end{equation}
since  $\Lambda^\circ=\frac{1}{\tau}B^n$ and thus
$h_\Lambda=j_{\Lambda^\circ}=\tau\|\cdot\|$.
Moreover, if $\Delta$ is strictly convex, then the action of
 any gliding $(A,\Delta, B^n)$-billiard trajectory
$\gamma:[0, T]\to \partial(\Delta\times B^n)$
 is also equal to the length of the projection $\pi_q(\gamma)$ in $\Delta$.
\end{proposition}

\begin{proof}
 Firstly,
we prove (\ref{e:linesymp4}) in the case that $0\in{\rm Int}(\Delta)$ and $0\in{\rm Int}(\Lambda)$.
By a direct computation  we have
\begin{eqnarray*}
A(\gamma)&=&\frac{1}{2}\int^T_0\langle -J\dot\gamma(t),\gamma(t)\rangle dt\\
&=&\frac{1}{2}\sum^{m}_{j=0}\int^{t_{j+1}}_{t_j}\langle -J\dot\gamma(t),\gamma(t)\rangle dt\\
&=&\frac{1}{2}\sum^{m}_{j=0}\int^{t_{j+1}}_{t_j}\left[(\dot{p}(t), q(t))_{\mathbb{R}^n}-(\dot{q}(t),p(t))_{\mathbb{R}^n}\right]dt\\
&=&-\sum^{m}_{j=0}\int^{t_{j+1}}_{t_j}(\dot{q}(t),{p}(t))_{\mathbb{R}^n}dt+
\frac{1}{2}\sum^{m}_{j=0}\left[(q(t_{j+1}),p(t_{j+1}))_{\mathbb{R}^n}-(q(t_j), p(t_j))_{\mathbb{R}^n}\right]\\
&=&-\sum^{m}_{j=0}\int^{t_{j+1}}_{t_j}(\dot{q}(t), {p}(t))_{\mathbb{R}^n}dt+
\frac{1}{2}\left[(q(t_{m+1}),p(t_{m+1}))_{\mathbb{R}^n}-(q(t_0),p(t_0))_{\mathbb{R}^n}\right]\\
&=&-\sum^{m}_{j=0}\int^{t_{j+1}}_{t_j}(\dot{q}(t), p(t))_{\mathbb{R}^n}dt
\end{eqnarray*}
since $(q(t_{m+1}),p(t_{m+1}))_{\mathbb{R}^n}=(Aq(t_0), (A^t)^{-1}p(t_0))_{\mathbb{R}^n}=
(q(t_0),p(t_0))_{\mathbb{R}^n}$. By (BT1) we have
\begin{eqnarray*}
-\int^{t_{i+1}}_{t_i}(\dot{q}(t),p(t))_{\mathbb{R}^n}dt=
-(q(t_{i+1})-q(t_i), p(t_i))_{\mathbb{R}^n}=-(q_{i+1}-q_i, p_i)_{\mathbb{R}^n},
\end{eqnarray*}
where $j_\Lambda(p_i)=1$ and $q_{i+1}-q_i=-\kappa (t_{i+1}-t_i)\nabla j_\Lambda(p_i)$.
The last two equalities mean that $-(q_{i+1}-q_i, p_i)_{\mathbb{R}^n}$ is
either the maximum or the minimum of the function
$p\mapsto -(q_{i+1}-q_i, p)_{\mathbb{R}^n}$ on $j^{-1}_\Lambda(1)$. Note that
\begin{eqnarray*}
-\int^{t_{i+1}}_{t_i}(\dot{q}(t),p(t))_{\mathbb{R}^n}dt=
\int^{t_{i+1}}_{t_i}(\kappa\nabla j_\Lambda(p(t_i)), p(t_i))_{\mathbb{R}^n}dt=
\kappa(t_{i+1}-t_i)>0.
\end{eqnarray*}
So $-(q_{i+1}-q_i, p_i)_{\mathbb{R}^n}$ must be the maximum of the function
$p\mapsto -(q_{i+1}-q_i, p)_{\mathbb{R}^n}$ on $j^{-1}_\Lambda(1)$,
which by definition equals $h_\Lambda(q_i-q_{i+1})$. In this case
(\ref{e:linesymp4}) follows immediately.

Next, we deal with  the general case. Now we have
 $\bar{q}\in{\rm Int}(\Delta)$ and $\bar{p}\in{\rm Int}(\Lambda)$ such that
the above result can be applied to $\gamma-(\bar{q},\bar{p})$ yielding
\begin{eqnarray*}
A(\gamma-(\bar{q},\bar{p}))&=&\sum^{m}_{j=0}h_{\Lambda-\bar{p}}((q_j-\bar{q})-(q_{j+1}-\bar{q}))=
\sum^{m}_{j=0}h_{\Lambda-\bar{p}}(q_j-q_{j+1})\\
&=&\sum^{m}_{j=0}h_{\Lambda}(q_j-q_{j+1})-\sum^{m}_{j=0}(\bar{p}, q_j-q_{j+1})_{\mathbb{R}^n}
\end{eqnarray*}
because $h_{\Lambda-\bar{p}}(u)=h_{\Lambda}(u)-(\bar{p}, u)_{\mathbb{R}^n}$,
where  $q_j=\pi_q(\gamma(t_i))$, $i=0,\cdots,m+1$, where $t_0=0$, $t_{m+1}=T$ and $q_{m+1}=Aq_0$.
Moreover, as above we may compute
\begin{eqnarray*}
A(\gamma)&=&-\sum^{m}_{j=0}\int^{t_{j+1}}_{t_j}(\dot{q}(t), p(t))_{\mathbb{R}^n}dt,\\
A(\gamma-(\bar{q},\bar{p}))&=&-\sum^{m}_{j=0}\int^{t_{j+1}}_{t_j}(\dot{q}(t), p(t)-\bar{p})_{\mathbb{R}^n}dt\\
&=&-\sum^{m}_{j=0}\int^{t_{j+1}}_{t_j}(\dot{q}(t), p(t))_{\mathbb{R}^n}dt-\sum^{m}_{j=0}(\bar{p}, q_j-q_{j+1})_{\mathbb{R}^n}
\end{eqnarray*}
These lead to the desired  (\ref{e:linesymp4}) directly.

Thirdly, we prove the final claim. Now $\bar{p}=0$, The above expressions show that
$A(\gamma)=A(\gamma-(\bar{q},0)$. Since $\pi_q(\gamma)-\bar{q}$ and $\pi_q(\gamma)$
have the same length, we only need to prove the case  $\bar{q}=0$.

Since $\gamma$ is gliding, by Proposition~\ref{prop:billT.0}(i) we have
$$
\dot\gamma(t)=(\dot{\gamma}_q(t), \dot{\gamma}_p(t))=(-\alpha(t)\gamma_p(t)/|\gamma_p(t)|,
\beta(t)\nabla g_\Delta(\gamma_q(t))),
$$
where $\alpha$ and $\beta$ are two smooth positive functions satisfying
a condition as in \cite[(8)]{AAO14}. Hence $\gamma_q=\pi_q(\gamma)$ has length
$$
L(\gamma_q)=\int^T_0|\dot{\gamma}_q(t)|dt=\int^T_0\alpha(t)dt.
$$
On the other hand, as above we have
\begin{eqnarray*}
A(\gamma)&=&\frac{1}{2}\int^T_0\langle-J\dot\gamma(t),\gamma(t)\rangle dt\\
&=&\frac{1}{2}\int^T_0\bigl((\dot\gamma_p(t),\gamma_q(t))_{\mathbb{R}^n}-
(\dot\gamma_q(t)\gamma_p(t)))_{\mathbb{R}^n}\bigr)dt\\
&=&-\int^T_0(\gamma_p(t),\dot\gamma_q(t)))_{\mathbb{R}^n}dt=\int^T_0\alpha(t)dt.
\end{eqnarray*}
\end{proof}

\section{Proofs of Theorems~\ref{th:billT.2},~\ref{th:billT.4} and Proposition~\ref{prop:add}}\label{sec:BillT-2}
\setcounter{equation}{0}

\begin{proof}[{\it Proof of Theorem~\ref{th:billT.2}}]
Let $\lambda\in (0,1)$. Since
${\rm Int}(\Delta_1)\cap{\rm Fix}(A)\ne\emptyset$,
 ${\rm Int}(\Delta_2)\cap{\rm Fix}(A)\ne\emptyset$
 and ${\rm Int}(\Lambda)\cap{\rm Fix}(A^t)\ne\emptyset$,
 ${\rm Fix}(\Psi_A)$ is intersecting with
 both ${\rm Int}(\Delta_1\times\Lambda)$ and ${\rm Int}(\Delta_2\times\Lambda)$.
Note that
\begin{eqnarray*}
&&\bigl(\lambda\Delta_1\bigr)\times \bigl(\lambda\Lambda\bigr)+
\bigl((1-\lambda)\Delta_2\bigr)\times \bigl((1-\lambda)\Lambda\bigr)\\
&=&\bigl(\lambda\Delta_1+(1-\lambda)\Delta_2\bigr)\times \bigl(\lambda\Lambda+
(1-\lambda)\Lambda\bigr)\\
&=&\bigl(\lambda\Delta_1+(1-\lambda)\Delta_2\bigr)\times \Lambda.
\end{eqnarray*}
It follows from Corollary~\ref{cor:Brun.1} that
\begin{eqnarray}\label{e:BillT.1}
&&\bigl(c^{\Psi_A}_{\rm EHZ}\bigl(\lambda\Delta_1\times \lambda\Lambda\bigr)\bigr)^{\frac{1}{2}}+
\bigl(c^{\Psi_A}_{\rm EHZ}\bigl((1-\lambda)\Delta_2\times (1-\lambda)\Lambda\bigr)\bigr)^{\frac{1}{2}}\nonumber\\
&\le &\bigl(c^{\Psi_A}_{\rm EHZ}\bigl(\bigl(\lambda\Delta_1+(1-\lambda)\Delta_2\bigr)\times \Lambda\bigr)\bigr)^{\frac{1}{2}},
\end{eqnarray}
which is equivalent  to
\begin{eqnarray}\label{e:BillT.2}
&&\lambda\bigl(c^{\Psi_A}_{\rm EHZ}\bigl(\Delta_1\times \Lambda\bigr)\bigr)^{\frac{1}{2}}+
(1-\lambda)\bigl(c^{\Psi_A}_{\rm EHZ}\bigl(\Delta_2\times \Lambda\bigr)\bigr)^{\frac{1}{2}}\nonumber\\
&\le &\bigl(c^{\Psi_A}_{\rm EHZ}\bigl(\bigl(\lambda\Delta_1+(1-\lambda)\Delta_2\bigr)\times \Lambda\bigr)^{\frac{1}{2}}.
\end{eqnarray}
By this and the weighted arithmetic-geometric mean inequality
\begin{eqnarray*}
&&\lambda\bigl(c^{\Psi_A}_{\rm EHZ}\bigl(\Delta_1\times \Lambda\bigr)\bigr)^{\frac{1}{2}}+
(1-\lambda)\bigl(c^{\Psi_A}_{\rm EHZ}\bigl(\Delta_2\times \Lambda\bigr)\bigr)^{\frac{1}{2}}\\
&&\ge \left(\bigl(c^{\Psi_A}_{\rm EHZ}\bigl(\Delta_1\times \Lambda\bigr)\bigr)^{\frac{1}{2}}\right)^\lambda
\left(\bigl(c^{\Psi_A}_{\rm EHZ}\bigl(\Delta_2\times \Lambda\bigr)\bigr)^{\frac{1}{2}}\right)^{(1-\lambda)},
\end{eqnarray*}
we get
\begin{eqnarray}\label{e:BillT.3}
&&\left(\bigl(c^{\Psi_A}_{\rm EHZ}\bigl(\Delta_1\times \Lambda\bigr)\bigr)^{\frac{1}{2}}\right)^\lambda
\left(\bigl(c^{\Psi_A}_{\rm EHZ}\bigl(\Delta_2\times \Lambda\bigr)\bigr)^{\frac{1}{2}}\right)^{(1-\lambda)}
\nonumber\\
&\le& \bigl(c^{\Psi_A}_{\rm EHZ}\bigl(\bigl(\lambda\Delta_1+(1-\lambda)\Delta_2\bigr)\times \Lambda\bigr)^{\frac{1}{2}}.
\end{eqnarray}
Replacing $\Delta_1$ and $\Delta_2$ by $\Delta_1':=\lambda^{-1}\Delta_1$ and $\Delta_2':=(1-\lambda)^{-1}\Delta_2$,
respectively, we arrive at
\begin{eqnarray}\label{e:BillT.4}
\left(\bigl(c^{\Psi_A}_{\rm EHZ}\bigl(\Delta'_1\times \Lambda\bigr)\bigr)^{\frac{1}{2}}\right)^\lambda
\left(\bigl(c^{\Psi_A}_{\rm EHZ}\bigl(\Delta'_2\times \Lambda\bigr)\bigr)^{\frac{1}{2}}\right)^{(1-\lambda)}
\le \bigl(c^{\Psi_A}_{\rm EHZ}\bigl(\bigl(\Delta_1+\Delta_2\bigr)\times \Lambda\bigr)^{\frac{1}{2}}.
\end{eqnarray}
For any $\mu>0$, since
$$
\phi:(\Delta_1\times\Lambda, \mu\omega_0)\to ((\mu\Delta_1)\times\Lambda, \omega_0),\;(x,y)\mapsto (\mu x,y)
$$
is a symplectomorphism which commutes with $\Psi_A$, we have
$$
c^{\Psi_A}_{\rm EHZ}\bigl(\Delta'_1\times \Lambda\bigr)=\lambda^{-1}c^{\Psi_A}_{\rm EHZ}\bigl(\Delta_1\times \Lambda\bigr),
\qquad c^{\Psi_A}_{\rm EHZ}\bigl(\Delta'_2\times \Lambda\bigr)=(1-\lambda)^{-1}c^{\Psi_A}_{\rm EHZ}\bigl(\Delta_2\times \Lambda\bigr).
$$
Let us choose $\lambda\in (0,1)$ such that $\Upsilon:=c^{\Psi_A}_{\rm EHZ}\bigl(\Delta'_1\times \Lambda\bigr)=c^{\Psi_A}_{\rm EHZ}\bigl(\Delta'_2\times \Lambda\bigr)$, i.e.,
\begin{equation}\label{e:BillT.5}
\lambda=\frac{c^{\Psi_A}_{\rm EHZ}(\Delta_1\times \Lambda)}{
c^{\Psi_A}_{\rm EHZ}(\Delta_1\times \Lambda)+ c^{\Psi_A}_{\rm EHZ}(\Delta_2\times \Lambda)}.
\end{equation}
Then
\begin{eqnarray}\label{e:BillT.6}
\xi^A_\Lambda(\Delta_1+\Delta_2)&=&c^{\Psi_A}_{\rm EHZ}\bigl(\bigl(\Delta_1+\Delta_2\bigr)\times \Lambda\bigr)\nonumber\\
&\ge&\left(c^{\Psi_A}_{\rm EHZ}\bigl(\Delta'_1\times \Lambda\bigr)\right)^\lambda
\left(c^{\Psi_A}_{\rm EHZ}\bigl(\Delta'_2\times \Lambda\bigr)\right)^{(1-\lambda)}\nonumber\\
&=&\Upsilon=\lambda\Upsilon+(1-\lambda)\Upsilon\nonumber\\
&=&\lambda c^{\Psi_A}_{\rm EHZ}\bigl(\Delta'_1\times \Lambda\bigr)+(1-\lambda)c^{\Psi_A}_{\rm EHZ}\bigl(\Delta'_2\times \Lambda\bigr)
\nonumber\\
&=&c^{\Psi_A}_{\rm EHZ}\bigl(\Delta_1\times \Lambda\bigr)+c^{\Psi_A}_{\rm EHZ}\bigl(\Delta_2\times \Lambda\bigr)\nonumber\\
&=&\xi^A_\Lambda(\Delta_1)+ \xi^A_\Lambda(\Delta_2)
\end{eqnarray}
and hence (\ref{e:linesymp6}) holds.

Final claim follows from Corollary~\ref{cor:Brun.1}. Theorem~\ref{th:billT.2} is proved.
\end{proof}

%
%

\begin{proof}[Proof of Proposition~\ref{prop:add}]
{\bf (i)} By the definition of $\xi^A$ and  Proposition~\ref{MonComf}(i)-(ii) we have
\begin{eqnarray}\label{e:BillT.8}
\xi^A(\Delta)&=&c^{\Psi_A}_{\rm EHZ}(\Delta\times B^n)\nonumber\\
&\ge& c^{\Psi_A}_{\rm EHZ}(B^{n}(\bar{q},r)\times B^n)\nonumber\\
&=&c^{\Psi_A}_{\rm EHZ}(B^{n}(0,r)\times B^n)
\end{eqnarray}
since $(\bar{q},0)$ is a fixed point of $\Psi_A$. Note that
\begin{eqnarray}\label{e:BillT.9}
B^n(0,r)\times B^n\to B^n(0,\sqrt{r})\times B^n(0,\sqrt{r}),\;(q,p)\mapsto (q/\sqrt{r}, \sqrt{r}p)
\end{eqnarray}
is a symplectomorphism which commutes with $\Psi_A$.
Using Proposition~\ref{MonComf}(i)-(ii) we deduce
\begin{eqnarray*}
c^{\Psi_A}_{\rm EHZ}(B^{n}(0,r)\times B^n)&=&c^{\Psi_A}_{\rm EHZ}(B^n(0,\sqrt{r})\times B^n(0,\sqrt{r}))\\
&=&rc^{\Psi_A}_{\rm EHZ}(B^n\times B^n)\\
&\ge &rc^{\Psi_A}_{\rm EHZ}(B^{2n})=\frac{r \mathfrak{t}(\Psi_A)}{2}
\end{eqnarray*}
because of (\ref{e:TPsi}). Then (\ref{e:linesymp8}) follows from  (\ref{e:BillT.8}).

\noindent{\bf (ii)} For any $u\in S^n_\Delta$, $\Delta$ sits between support planes $H(\Delta,u)$ and $H(\Delta,-u)$,
 and the hyperplane $H_u$  is between  $H(\Delta,u)$ and $H(\Delta,-u)$ and
 has distance ${\rm width}(\Delta)/2$ to $H(\Delta,u)$ and $H(\Delta,-u)$ respectively.
Obverse that $\Psi_{{\bf O}, \bar{q}}(\Delta\times B^n)=({\bf O}(\Delta-\bar{q}))\times B^n$ is contained in
$Z^{2n}_\Delta$. From this and (\ref{e:inv}) it follows that
\begin{eqnarray*}
\xi^A(\Delta)=c^{\Psi_A}_{\rm EHZ}(\Delta\times B^n)
=c_{\rm EHZ}^{\Psi_{{\bf O}, \bar{q}}\Psi_A\Psi_{{\bf O}, \bar{q}}^{-1}}(
\Psi_{{\bf O}, \bar{q}}(\Delta\times B^n))
\le  c_{\rm EHZ}^{\Psi_{{\bf O}, \bar{q}}\Psi_A\Psi_{{\bf O}, \bar{q}}^{-1}}(Z^{2n}_\Delta).
\end{eqnarray*}
Hence (\ref{e:linesymp17}) is proved.
\end{proof}

In order to prove Theorem~\ref{th:billT.4} we need:

\begin{lemma}\label{lem:add}
For $A\in{\rm GL}(n)$ and a convex body $\Delta\subset\mathbb{R}^n_q$
with ${\rm Fix}(A)\cap{\rm Int}(\Delta)\ne\emptyset$,
if $\Delta$ is contained in the closure of the ball $B^{n}(\bar{q},R)$
with $A\bar{q}=\bar{q}\in {\rm Int}(\Delta)$, then
\begin{equation}\label{e:linesymp8.2}
\xi^A(\Delta)\le \mathfrak{t}(\Psi_A)R.
\end{equation}
\end{lemma}
\begin{proof}
 As in the proof of Proposition~\ref{prop:add}(i) we deduce
\begin{eqnarray*}
\xi^A(\Delta)&=&c^{\Psi_A}_{\rm EHZ}(\Delta\times B^n)\\
&\le& c^{\Psi_A}_{\rm EHZ}(B^{n}(\bar{q},R)\times B^n)\nonumber\\
&=&c^{\Psi_A}_{\rm EHZ}(B^{n}(0,R)\times B^n)\\
&=&c^{\Psi_A}_{\rm EHZ}(B^{n}(0,\sqrt{R})\times B^n(0,\sqrt{R}))\\
&=&Rc^{\Psi_A}_{\rm EHZ}(B^{n}\times B^n)\\
&\le &Rc^{\Psi_A}_{\rm EHZ}(B^{2n}(0,\sqrt{2}))\le  \mathfrak{t}(\Psi_A)R
\end{eqnarray*}
by (\ref{e:TPsi}). This and Theorem~\ref{th:convex} yield the desired claims.
\end{proof}

\begin{proof}[Proof of Theorem~\ref{th:billT.4}]
Under the assumptions of Theorem~\ref{th:billT.4} it was stated in the bottom of \cite[page 177]{AAO14} that
$\xi(\Delta)=L(\sigma)$  for some periodic billiard trajectory $\sigma$ in $\Delta$.
 It follows from Lemma~\ref{lem:add} that
$\xi(\Delta)=\xi^{I_n}(\Delta)\le \pi{\rm diam}(\Delta)$, and so
$L(\sigma)\le \pi{\rm diam}(\Delta)$.
\end{proof}

\vspace{5mm}
\noindent{\bf Declarations}\\

\noindent{\bf Data Availability Statements Non applicable.}\\

\noindent{\bf Conflict of interest}\quad The authors declare that they have no conflict of interest.


\medskip

\begin{tabular}{l}
Department of Mathematics, Civil Aviation University of China\\
 Tianjin  300300, The People's Republic of China\\
 E-mail address: rrjin@cauc.edu.cn\\
 \\
 School of Mathematical Sciences, Beijing Normal University\\
 Laboratory of Mathematics and Complex Systems, Ministry of Education\\
 Beijing 100875, The People's Republic of China\\
 E-mail address: gclu@bnu.edu.cn\\
\end{tabular}
\end{document}